\documentclass[10pt]{amsart}

\pdfoutput=1
\usepackage[text={420pt,660pt},centering]{geometry}

\usepackage{esint,amssymb} 
\usepackage{graphicx}
\usepackage{MnSymbol}
\usepackage{mathtools} 
\usepackage[colorlinks=true, pdfstartview=FitV, linkcolor=blue, citecolor=blue, urlcolor=blue,pagebackref=false]{hyperref} \usepackage{microtype}
\usepackage{bm}
\usepackage{dsfont}
\usepackage{enumitem}

\newtheorem{proposition}{Proposition}
\newtheorem{theorem}[proposition]{Theorem}
\newtheorem{lemma}[proposition]{Lemma}
\newtheorem{corollary}[proposition]{Corollary}

\newtheorem{remark}[proposition]{Remark}
\newtheorem{example}[proposition]{Example}
\theoremstyle{definition}
\newtheorem{definition}[proposition]{Definition}

\numberwithin{equation}{section}
\numberwithin{proposition}{section}
\numberwithin{figure}{section}
\numberwithin{table}{section}

\newcommand{\R}{\mathbb{R}}
\newcommand{\N}{\mathbb{N}}

\renewcommand{\d}{\mathrm{d}}
\newcommand{\la}{\left\langle}\newcommand{\ra}{\right\rangle}
\newcommand{\eps}{\epsilon}
\newcommand{\cX}{\mathcal{X}}

\renewcommand{\H}{\mathsf{H}}

\newcommand{\cH}{\mathcal{H}}

\newcommand{\cD}{\mathcal{D}}

\newcommand{\cE}{\mathcal{E}}

\renewcommand{\bar}{\overline}
\renewcommand{\tilde}{\widetilde}

\newcommand{\sF}{\mathsf{F}}

\newcommand{\dist}{\mathsf{dist}}

\newcommand{\itr}{\mathsf{int}\,}

\newcommand{\C}{\mathcal C}

\usepackage{ulem}
\usepackage{cancel}

\newcommand{\rv}[1]{\textcolor{black}{#1}}

\newcommand{\HJ}{\mathrm{HJ}}

\renewcommand{\hat}{\widehat}

\newcommand{\supp}{\mathsf{supp}\,}

\begin{document}

\author[Hong-Bin Chen]{Hong-Bin Chen}
\address[Hong-Bin Chen]{Courant Institute of Mathematical Sciences, New York University, New York, New York, USA}
\email{hbchen@cims.nyu.edu}

\author[Jiaming Xia]{Jiaming Xia}
\address[Jiaming Xia]{Department of Mathematics, University of Pennsylvania, Philadelphia, Pennsylvania, USA}
\email{xiajiam@sas.upenn.edu}

\keywords{Hamilton--Jacobi equation, monotonicity, convex cone}
\subjclass[2020]{35A01, 35A02, 35D40, 35F21}

\title{Hamilton--Jacobi equations with monotone nonlinearities on convex cones}

\begin{abstract}
We study the Cauchy problem of a Hamilton--Jacobi equation with the spatial variable in a closed convex cone. A monotonicity assumption on the nonlinearity allows us to prescribe no condition on the boundary of the cone. We show the well-posedness of the equation in the viscosity sense and prove several properties of the solution: monotonicity, Lipschitzness, and representations by variational formulas.
\end{abstract}

\maketitle

\section{Introduction}

In \cite{mourrat2018hamilton, mourrat2020nonconvex, mourrat2020free, chen2020hamiltonTensor, chen2021statistical, chen2022hamilton}, the limit free energy of mean-field disordered models has been shown to be described by the viscosity solution to a Hamilton--Jacobi equation of the following form:
\begin{align}\label{e.hj_intro}
    \begin{cases}
    \partial_t f - \H(\nabla f) =0 ,\quad & \text{on $\R_+\times \C$}
    \\
    f(0,\cdot)=\psi  ,\quad & \text{on $ \C$}
    \end{cases}
\end{align}
where we set $\R_+ = [0,\infty)$ throughout and $\C$ is a finite-dimensional closed convex cone. Moreover, in these settings, the restriction of the nonlinearity $\H$ to $\C$, and the initial condition $\psi$ are monotone along the direction given by the dual cone $\C^*$ of $\C$. In this work, we interpret the well-posedness as the comparison principle together with the existence of solutions.

Other than invalidating the Dirichlet boundary condition, the settings of these models do not suggest a suitable boundary condition on $\partial\C$. To ensure the well-posedness of the Cauchy problem in the sense of viscosity solutions, the first attempt was to impose the Neumann boundary condition as done in \cite{mourrat2018hamilton, mourrat2020nonconvex, mourrat2020free}. When the geometry of $\partial \C$ is complicated, for example, in the case where $\C$ is the set of positive semi-definite matrices, the Neumann boundary condition makes it hard to verify the comparison principle and the existence of solutions at the same time. 

Hence, in \cite{chen2020hamiltonTensor, chen2021statistical}, another attempt was to introduce a modification $\bar \H$ that coincides with $\H$ on $\C$ but is monotone on the entire space and to require the solution to satisfy the equation in the viscosity sense also on $\partial \C$. This approach turns out to be successful. However, the modification $\bar \H$ seems artificial, and it was not shown whether the solution depends on the way to construct the modification. 

In this paper, we show that there is no need to impose any boundary condition on $\partial \C$, due to the monotonicity of the nonlinearity and the initial condition. The solutions are required to satisfy the equation in the viscosity sense only on the interior of $\R_+\times\C$. Moreover, only the values of $\H$ on $\C$ matter, and thus the solution does not depend on how $\H$ is modified outside $\C$.
Aside from the well-posedness, 
we also prove several useful properties of the solution that were needed in the aforementioned works and will be needed in future works.
In particular, the results here for finite-dimensional equations will be used in the study of the equation in infinite dimensions~\cite{chen2022hamilton}.

The main idea behind the irrelevance of the boundary is from \cite{crandall1985viscosity,souganidis1986remark}:
one can use $\mathsf{dist}(\cdot,\partial \C)$, the distance function to the boundary, to push the contact points of test functions away from the boundary. In our setting, this function admits a representation~\eqref{e.d_c=dist} better reflecting the geometry of cones.

\subsection{Definitions and the main result}\label{s.intro_main_result}

Let $\cH$ be a finite-dimensional Hilbert space with inner product $\la\, \cdot\,,\,\cdot\,\ra$ and associated norm $|\cdot|$. For a subset $\mathcal{A}\subset\cH$, we denote its interior by $\itr \mathcal{A}$ or $\mathring{\mathcal{A}}$.

\subsubsection{Cones}
Throughout, we assume that
\begin{enumerate}[start=1,label={\rm{(A)}}]
    \item\label{i.assump_C} $\C$ is a nonempty closed convex cone in $\cH$, which is pointed (i.e.\ $\C\cap(-\C)=\{0\}$) and has nonempty interior (i.e.\ $\mathring{\C}\neq\emptyset$).
\end{enumerate}
Here, $\C$ is said to be a convex cone if $\lambda x + \lambda' x'\in \C$ for all $\lambda,\lambda'\geq 0$ and $x,x'\in\C$.
Note that being pointed means that $\C$ does not contain any nonzero subspace. By restricting to a subspace of $\cH$, $\mathring\C$ can always be made nonempty.

The dual cone of $\C$ is defined by $\C^* = \{x\in\cH:\la x, y\ra \geq 0, \, \forall y \in \C\}$.
Since $\C$ is a nonempty closed convex cone, it is easy to see that $\C^*$ is also a nonempty closed convex cone. Moreover, it is well-known (c.f.\ \cite[Corollary~6.33]{bauschke2011convex}) that $(\C^*)^* = \C$.

Since a cone can induce an order, we define the following notion of monotonicity.
For $\cD\subset\cH$ and $g:\cD\to (-\infty,\infty]$, $g$ is said to be \textit{$\C^*$-increasing} if $g(x)\geq g(y)$ for all $x,y\in\cD$ satisfying $x- y \in \C^*$.
Notice that if $g$ is a differentiable function on $\C$ and $g$ is $\C^*$-increasing, then $\nabla g(x)\in \C$ for every $x\in\mathring\C$. Indeed, $g(x+\eps y)\geq g(x)$ for every $y\in\C^*$ and $\eps>0$ sufficiently small. Sending $\eps\to0$, we have $\la y, \nabla g(x)\ra\geq 0$ and thus $\nabla g(x)\in(\C^*)^*=\C$.

\subsubsection{Viscosity solutions}
Recall that we have set $R_+=[0,\infty)$. For a subset $\cD\subset \cH$ and a function $\sF:\cH\to\R$, we consider the following Hamilton--Jacobi equation, as a slight generalization of \eqref{e.hj_intro}:
\begin{align*}
    \partial_t f - \sF(\nabla f)=0,\quad\text{on}\ \R_+\times \cD.
\end{align*}
We denote this equation by $\HJ(\cD, \sF)$ and the associated Cauchy problem with initial condition $f(0,\cdot)=\psi$ by $\HJ(\cD,\sF;\psi)$. 
Classical references to viscosity solutions include \cite{crandall1983viscosity,crandall1984some,lions1982generalized,crandall1984developments}.

\begin{definition}[Viscosity solutions]\label{d.vs}
\leavevmode
\begin{enumerate}
    \item An upper semi-continuous function $f:\R_+\times \cD\to \R$ is a \textit{viscosity subsolution} of~$\HJ(\cD, \sF)$ if for every $(t,x) \in (0,\infty)\times \cD$ and every smooth $\phi:(0,\infty)\times \cD\to\R$ such that $f-\phi$ has a local maximum at $(t,x)$, we have
\begin{align*}
\left(\partial_t \phi - \sF(\nabla\phi)\right)(t,x)\leq 0.
\end{align*}

\item A lower semi-continuous function $f:\R_+\times \cD\to \R$ is a \textit{viscosity supersolution} of~$\HJ(\cD, \sF)$ if for every $(t,x) \in (0,\infty)\times \cD$ and every smooth $\phi:(0,\infty)\times \cD\to\R$ such that $f-\phi$ has a local minimum at $(t,x)$, we have
\begin{align*}
\left(\partial_t \phi - \sF(\nabla\phi)\right)(t,x)\geq 0. 
\end{align*}

\item A continuous function $f:\R_+\times \cD\to \R$ is a \textit{viscosity solution} of~$\HJ(\cD, \sF)$ if $f$ is both a viscosity subsolution and supersolution.

\end{enumerate}

\end{definition}

Here, $f-\phi$ is said to have a local extremum at $(t,x)$ if $f-\phi$ achieves an extremum at $(t,x)$ over an open neighborhood of $(t,x)$ (in $\R\times\cH$) intersected with $(0,\infty)\times\cD$. 

We often drop the qualifier ``viscosity'', and simply write that $f$ is a subsolution, supersolution, or solution. We also write that $f$ solves $\HJ(\cD,\sF)$ if $f$ is a solution.
Throughout, the restriction of a function $g$ to a set $\cE$ is denoted by $g \lfloor_\cE$. For $\cD'\supset\cD$, a function $f:\R_+\times \cD'\to\R$ is said to be a subsolution, supersolution, or solution of $\HJ(\cD, \sF)$ if $f\lfloor_{\R_+\times\cD}$ is so. 

The goal of this work is to show that the solution $f:\R_+\times\C\to\R$ of \eqref{e.hj_intro} should be understood to be the solution of $\HJ(\mathring\C,\H)$. In other words, the values on $\partial\C$ are not relevant.

\subsubsection{Classes of functions}

For a metric space $\cX$, we define $\Gamma_\mathrm{cont}(\cX)$, $\Gamma_\mathrm{Lip}(\cX)$, and $\Gamma_\mathrm{locLip}(\cX)$ to be the class of functions $g:\cX\to\R$ that are continuous, Lipschitz, and locally Lipschitz, respectively. For $\cD\subset \cH$, we denote by $\Gamma^\nearrow(\cD)$ the class of $\C^*$-increasing functions on $\cD$. We also set $\Gamma^\nearrow_\square(\cD)= \Gamma^\nearrow(\cD)\cap \Gamma_\square(\cD)$ where $\square$ is a placeholder for subscripts.

We define
\begin{align*}
    \mathfrak{M} = \left\{f:\R_+\times \C\to\R\ \big|\ f(t,\cdot)\in \Gamma^\nearrow(\C),\,\forall t\in\R_+\right\},
\end{align*}
and
\begin{align*}
    \mathfrak{L} =\left\{f:\R_+\times \C\to\R\ \bigg|\ f(0,\cdot)\in \Gamma_\mathrm{Lip}(\C);\quad \sup_{t>0,\,x\in\C}\frac{|f(t,x) - f(0,x)|}{t}<\infty\right\}.
\end{align*}
We define the following subclass
\begin{align*}
    \mathfrak{L}_\mathrm{Lip} = \left\{f\in\mathfrak{L}\ \bigg| \ \sup_{t\in\R_+}\|f(t,\cdot)\|_\mathrm{Lip}<\infty\right\},
\end{align*}
where we denote by $\|\cdot\|_\mathrm{Lip}$ the Lipschitz coefficient of a function. Note that functions in $\mathfrak{M}$, $\mathfrak{L}$ or $\mathfrak{L}_\mathrm{Lip}$ are not required to be continuous.

We are mainly interested in $\HJ(\mathring\C,\H)$, for $\H:\cH\to\R$ satisfying $\H\lfloor_\C\in \Gamma_\mathrm{locLip}^\nearrow(\C)$. Throughout, for $x\in\cH$ and $r>0$, we set $B(x,r)=\{y\in\cH:|y-x|\leq r\}$.

\begin{theorem}\label{t}
Let $\H:\cH\to\R$ satisfy $\H\lfloor_\C\in \Gamma^\nearrow_\mathrm{locLip}(\C)$. Then the following holds:
\begin{enumerate}
  \setlength\itemsep{1em}
    \item \label{i.main_cp} 
    (Comparison principle)
    
        \medskip
        
    \noindent If 
    $u,v:\R_+\times\C\to\R$ are respectively a subsolution and a supersolution of $\HJ(\mathring\C,\H)$ in $\mathfrak{M}\cap\mathfrak{L}_\mathrm{Lip}$ (or in $\mathfrak{M}\cap\mathfrak{L}$ if $\H\lfloor_\C\in \Gamma^\nearrow_\mathrm{Lip}(\C)$), then $\sup_{\R_+\times\C}(u-v)  = \sup_{\{0\}\times \C}(u-v)$.
    
    \item \label{i.main_exist}
    (Existence of solutions)
    
        \medskip
        
    \noindent For every $\psi\in \Gamma^\nearrow_\mathrm{Lip}(\C)$, there is a viscosity solution $f:\R_+\times \C\to\R$ of $\HJ(\mathring\C,\H;\psi)$ unique in $\mathfrak{M}\cap\mathfrak{L}_\mathrm{Lip}$ (or in $\mathfrak{M}\cap\mathfrak{L}$ if $\H\lfloor_\C\in\Gamma^\nearrow_\mathrm{Lip}(\C)$). In addition, $f$ satisfies the following:
    \bigskip
    \begin{enumerate}
    \setlength\itemsep{1em}
        \item  \label{i.main_lip} 
        (Lipschitzness)
        
        \medskip
        
        The solution $f$ is Lipschitz and satisfies
        \begin{gather*}
\sup_{t\in\R_+}\|f(t,\cdot)\|_\mathrm{Lip}=\|\psi\|_\mathrm{Lip},
            \\
            \sup_{x\in\C}\|f(\cdot,x)\|_\mathrm{Lip}\leq \sup_{\C\cap B(0,\|\psi\|_\mathrm{Lip})}|\H|.
        \end{gather*}
        
        \item
        \label{i.main_t_nond}
        (Monotonicity in time)
        
        \medskip
        
        If $\H\lfloor_\C \geq 0$, then $f(t,x)\leq f(t',x)$ for all $t'\geq t\geq 0$ and $x\in\C$.
  
        \item \label{i.main_equiv} 
        (Solving modified equations)
        
        \medskip
        
        For every $\sF\in\Gamma^\nearrow_\mathrm{locLip}(\cH)$ satisfying $\sF\lfloor_{\C\cap B(0,\|\psi\|_\mathrm{Lip})} = \H\lfloor_{\C\cap B(0,\|\psi\|_\mathrm{Lip})}$, $f$ is the solution of $\HJ(\C,\sF;\psi)$ unique in $\mathfrak{L}_\mathrm{Lip}$ (or in $\mathfrak{L}$ if $\sF\in\Gamma^\nearrow_\mathrm{Lip}(\cH)$).
        
        \item
        \label{i.main_var_rep}
        (Variational representations)
        
        \medskip
        
        Under an additional assumption that $\C$ has the {\rm Fenchel--Moreau property}, $f$ can be represented by the Hopf--Lax formula \eqref{e.hopf_lax} if $\H\lfloor_\C$ is convex and bounded below; and by the Hopf formula \eqref{e.hopf} if $\psi$ is convex.
    \end{enumerate}
\end{enumerate}
\end{theorem}

Given another norm on $\cH$ comparable to $|\cdot|$, Propositions~\ref{p.lipschitz} and~\ref{p.lip_t} show that the Lipschitz coefficients of $f$ with respect to this norm are also bounded, which is a generalization of \eqref{i.main_lip}. In~\eqref{i.main_equiv}, the definition of the \textit{Fenchel--Moreau property} is given in Definition~\ref{d.fenchel_moreau_prop}, which states that a version of the Fenchel--Moreau biconjugation identity holds for functions defined on $\C$. 

Since one does not need to pay attention to whether a test function touches the solution on the boundary or not, the analysis of
$\HJ(\C,\sF)$ is easier, in particular, in using the trick of doubling variables. Hence, the property \eqref{i.main_equiv} can be very useful in deriving further properties if needed in the future. In fact, many of the parts of the theorem are derived first for $\HJ(\C,\sF)$ and then transferred to $\HJ(\mathring\C,\H)$.

This theorem assembles results from various parts of the paper.

\begin{proof}
Part~\eqref{i.main_cp} follows from Corollary~\ref{c.cp_C} with assumption~\ref{i.case2_C} (or assumption~\ref{i.case1_C} if $\H\lfloor_\C\in \Gamma^\nearrow_\mathrm{Lip}(\C)$).

Let us verify Part~\eqref{i.main_exist}. Applying Lemma~\ref{l.ext_loclip}, we can find $\sF\in \Gamma^\nearrow_\mathrm{Lip}(\cH)$ equal to $\H$ on $B(0,\|\psi\|_\mathrm{Lip})$. Proposition~\ref{p.exist_sol_H} yields a solution $f\in\mathfrak{L}$ of $\HJ(\mathring\C,\sF;\psi)$. Proposition~\ref{p.drop_bdy} implies that $f$ solves $\HJ(\C,\sF;\psi)$. Corollary~\ref{c.lipschitz} implies that $f$ satisfies \eqref{i.main_lip}, and thus $f\in\mathfrak{L}_\mathrm{Lip}$. Then, by Proposition~\ref{p.monotone}, $f$ belongs to $\mathfrak{M}$.
These along with Lemma~\ref{l.equiv} imply that $f$ solves $\HJ(\mathring\C,\H;\psi)$. The uniqueness in the main statement of Part~\eqref{i.main_exist} is ensured by Part~\eqref{i.main_cp}. Hence, the main statement of Part~\eqref{i.main_exist}, and \eqref{i.main_lip} have been verified.

If $\H\lfloor_\C\geq0$, we have $\sF\geq 0$ by Lemma~\ref{l.ext_loclip}. From Proposition~\ref{p.monotone_time}, we can deduce \eqref{i.main_t_nond}. For \eqref{i.main_equiv}, Lemma~\ref{l.equiv} and Proposition~\ref{p.drop_bdy} imply that $f$ solves $\HJ(\C,\sF;\psi)$, and Corollary~\ref{c.cp_interior} with assumption~\ref{i.case2_CF} (or assumption~\ref{i.case1_CF} if $\sF\in\Gamma^\nearrow_\mathrm{Lip}(\cH)$) ensures the uniqueness. Lastly, \eqref{i.main_var_rep} follows from Propositions~\ref{p.hopf-lax} and~\ref{p.hopf}.
\end{proof}

We present a concrete setting, to which Theorem~\ref{t} is applicable.
\begin{example}\label{example}
Let $D\in\N$ and let $\cH$ be the linear space of $D\times D$ real symmetric matrices. We equip $\cH$ with the Frobenius inner product, i.e., $\la a, b\ra = \sum_{k,k'=1}^D a_{kk'}b_{kk'}$ for $a,b\in\cH$. Let $\C$ be the cone of positive semi-definite matrices in $\cH$. Consider $\H(a)=|a|^2$ for all $a\in\cH$, whose gradient at $a$ is $2a$. Hence, on $\C$, the gradient of $\H$ belongs to $\C$ and thus $\H\lfloor_\C$ is $\C^*$-increasing.
\end{example}

\subsection{Organization of the paper}
We collect preliminary results in~Section~\ref{s.prelim} including a simplification of the boundary condition on $\partial \C$ given that the nonlinearity is monotone, and modifications of the nonlinearity. These results allow us to equivalently study $\HJ(\C,\sF)$ for $\sF$, a modification of $\H$, that possesses good properties on the entirety of $\cH$, instead of $\HJ(\mathring\C,\H)$ as concerned in Theorem~\ref{t}. Additionally, since no special treatment is needed for points on $\partial\C$, the argument to analyze $\HJ(\C,\sF)$ via the trick of doubling the constants is simpler. Henceforth, we will focus on $\HJ(\C,\sF)$, and transfer its properties to $\HJ(\mathring\C,\H)$ using results from Section~\ref{s.prelim}, as done in the proof of Theorem~\ref{t}.

For $\HJ(\C,\sF)$, we show its well-posedness in Section~\ref{s.well-posed}; the monotonicity of the solution in Section~\ref{s.monotone}; the Lipschitzness of the solution in Section~\ref{s.Lip}; and variational formulae for the solution in Section~\ref{s.var}.

\subsubsection{Additional notation and convention}
Throughout, for $a,b\in\R$, we write $a\vee b = \max\{a,b\}$. For $(t,x)$ varying in a subset of $\R\times\cH$, we refer to $t$ as the \textit{temporal} variable and $x$ as the \textit{spatial} variable. The derivative in the temporal variable and the derivative in the spatial variable are denoted by $\partial_t$ and $\nabla$, respectively. We often abbreviate the word ``respectively'' as ``resp.'' The symbol $\sF$ is always reserved for a map from $\cH$ to $\R$ that satisfies good properties on $\cH$, and $\H$ is for a map from $\cH$ to $\R$ that enjoys good properties only when restricted to $\C$.

\subsection*{Acknowledgement}
The authors thank Jean-Christophe Mourrat for pointing out important references \cite{crandall1985viscosity,souganidis1986remark} and for helpful comments. The authors thank Tomas Dominguez for finding a mistake in the application of Perron's method.
The authors thank the anonymous referee for comments that helped them greatly improve the manuscript.

\section{Preliminary results}\label{s.prelim}

We first show in Proposition~\ref{p.drop_bdy} that if a nonlinearity $\sF$ is $\C^*$-increasing, then a solution of $\HJ(\mathring\C, \sF)$ is a solution of $\HJ(\C,\sF)$. Note that a solution of $\HJ(\C,\sF)$ is obviously a solution of $\HJ(\mathring\C,\sF)$. Hence, given a monotone nonlinearity, the values on the boundary are not relevant for the study of $\HJ(\C,\sF)$. Hence, to see if $f$ solves $\HJ(\C,\sF)$, we only need to consider $(t,x)\in(0,\infty)\times\mathring\C$ at which $f-\phi$ achieves a local extremum for some smooth $\phi$. The condition that $x\in\mathring\C$ allows us to better control $\nabla \phi(t,x)$ using properties of $f$, which is very useful in other sections.

Next, we record results on modifications of nonlinearities. Note that in Theorem~\ref{t}, the nonlinearity $\H$ is only assumed to have good properties on $\C$. For our analysis in later sections, it is useful to modify $\H$ into $\sF$ possessing these properties on $\cH$. Lemma~\ref{l.ext_loclip} is the main one to be used. Although not to be used here, Lemma~\ref{l.mod_grow} gives a good modification if $\H$ grows to infinity, which is useful in many applications. Lastly, Lemma~\ref{l.equiv} gives conditions under which the solution of the modified equation is the same solution of the original one and vice versa, which is needed to show Theorem~\ref{t}~\eqref{i.main_equiv} (see its proof in Section~\ref{s.intro_main_result}).

\subsection{Simplifying boundary}

\begin{proposition}\label{p.drop_bdy}
Let $\sF\in \Gamma^\nearrow_\mathrm{cont}(\cH)$. If $f:\R_+\times\C\to\R$ is a continuous subsolution (resp., supersolution) of $\HJ(\mathring\C,\sF)$, then $f$ is a subsolution (resp., supersolution) of $\HJ(\C,\sF)$.
\end{proposition}

As aforementioned, the distance function $\mathsf{dist} (\cdot,\partial \C)$ from \cite{crandall1985viscosity,souganidis1986remark} has a nice representation in our setting:
\begin{align}\label{e.d}
    d_\C(x) = \inf_{y\in\C^*,\ |y|=1}\la y, x\ra,\quad x \in \C.
\end{align}
We will directly work with this expression.
In an early version, we guessed but did not know how to prove
\begin{align}\label{e.d_c=dist}
    d_\C(x) = \dist (x,\partial \C),\quad  x\in \C.
\end{align}
In the review process, the referee provided a proof.
We briefly present it here.
\begin{proof}[Proof of~\eqref{e.d_c=dist}]
Fix any $x\in\C$. Let $y_0$ be the minimizer in~\eqref{e.d} and let $z$ be the projection of $x$ to $\{z'\in\cH:\la z',y_0\ra=0\}$.
Due to $|y_0|=1$, we can verify $x-z=\la y_0,x\ra y_0$. Now, for every $y\in \C^*$ with $|y|=1$, we have $\la y,z\ra=\la y,x\ra -\la y_0,x\ra\la y,y_0\ra \geq 0$, where we used the minimality of $y_0$ and $\la y,y_0\ra\leq 1$. Hence, $z\in \C$. Due to $\la z, y_0\ra=0$ and $y_0\in\C^*$, we can verify $z\in\partial \C$. Since $|x-z|=\la y_0,x\ra$, we get $\dist(x,\partial\C)\leq d_\C(x)$.
We turn to the other direction. Choose $z_0\in \partial \C$ to satisfy $|x-z_0| = \dist(x,\partial \C)$. There must be some $y\in\C^*$ with $|y|=1$ such that $\la y,z_0\ra=0$ (because otherwise one can show $z_0\in\mathring\C$). Now, $\la y,x\ra=\la y,x-z_0\ra \leq |x-z_0|$, which implies $d_\C(x) \leq \dist(x,\partial\C)$.
\end{proof}

Before stating properties of $d_\C$, we need some definition.
For a function $g:\cD\to\R$ defined on a subset $\cD\subset\cH$, the superdifferential of $g$ at $x\in\mathring\cD$, denoted by $D^+g(x)$, is defined to be the set of all $p\in \cH$ satisfying
\begin{align*}
    g(y) - g(x) \leq \la p, y-x\ra + o(y-x)
\end{align*}
as $y\to x$ within $\cD$.

\begin{lemma}\label{l.d}
The following holds for $d=d_\C$ defined in \eqref{e.d}:
\begin{enumerate}
\item\label{i.d(x)=0} for $x\in\C$, $d(x)=0$ if and only if $x\in \partial \C$;

\item\label{i.d_concave} $d$ is concave and Lipschitz;
\item\label{i.Dd} $D^+d(x)\subset \C^*\cap B(0,1)$ for every $x\in\mathring\C$;

\item\label{i.g-1/d} if $x\mapsto g(x) - \frac{1}{d(x)}$ achieves a local maximum at $x_0\in \mathring \C$ for some smooth function $g$, then $-(d(x_0))^2\nabla g(x_0)\in D^+d(x_0)$.
\end{enumerate}
\end{lemma}

\begin{proof}
Part~\eqref{i.d(x)=0}.
Let $x\notin \partial \C$. Then, there is $\eps>0$ such that $x-\epsilon y\in \C$ for all $y\in \C^*$ satisfying $|y|=1$. Since $\C = (\C^*)^*$, we have that, for all $y\in \C^*$ satisfying $|y|=1$,
\begin{align*}
    \la y,x\ra-\eps = \la y, x-\epsilon y\ra \geq 0,
\end{align*}
implying that $d(x)>0$. 

For the other direction, assume $d(x)>\eps$ for some $\eps>0$. Then, for any $y\in\C^*$ satisfying $|y|=1$ and any $y'\in B(0,1)$, we have $\la y, x\ra>\eps$ and then
\begin{align*}
    \la y, x-\eps y'\ra \geq \la y, x\ra -\eps>0.
\end{align*}
This implies that $x-\eps y'\in \C$ for all $y'\in B(0,1)$, and thus $x\notin\partial\C$.

Part \eqref{i.d_concave}.
As an infimum over linear functions, $d$ is concave.
Since $\cH$ is finite-dimensional, the infimum in $d$ is achieved. Let $x_1,x_2\in\C$. Then, there is $y_1\in \C^*$ with $|y_1|=1$ such that
\begin{align*}
    d(x_1)=\la y_1,x_1\ra \leq d(x_2) + \la y_1, x_1-x_2\ra\leq d(x_2)+|x_1-x_2|
\end{align*}
implying that $d$ is Lipschitz with coefficient $1$.

Part \eqref{i.Dd}. 
We first notice that $d(x)$ is $\C$-increasing (meaning that $d(x)\leq d(x')$ if $x'-x\in\C$). Fix any $x\in\mathring\C$. Since $d$ is concave, we know that $D^+d(x)$ is nonempty. Let $p\in D^+ d(x)$ be arbitrary. For any $z\in\C$ and $\epsilon>0$ small enough, we have $x+\epsilon z \in\C$ and
\begin{align*}
    \la p, \epsilon z\ra \geq d(x+\epsilon z)-d(x)-o(\epsilon |z|)\geq -o(\epsilon |z|).
\end{align*}
Therefore, sending $\epsilon\rightarrow 0$, we obtain $\la p,z\ra\geq 0$ for all $z\in \C$, which implies that $p\in \C^*$ by duality and thus $D^+d(x)\subset \C^{*}$. 

Then, we show that $D^+d(x) \subset B(0,1)$. Recall that we have shown in \eqref{i.d_concave} that $d$ is Lipschitz with coefficient $1$, which implies that
\begin{align*}
    \la p, -\epsilon z \ra+o(-\epsilon |z|)&\geq d(x-\epsilon z)-d(x)\geq -\epsilon|z|,
\end{align*}
for all $z\in B(0,1)$, $p\in D^+ d(x)$ and $\epsilon>0$ small.
Sending $\eps\to0$, we have $\la p, z\ra\leq |z|$. Therefore $D^+d(x)\subset B(0,1)$.

Part~\eqref{i.g-1/d}. 
Due to $x_0\in\mathring \C$, we have $x_0+z\in\C$ for sufficiently small $z$.
By maximality at $x_0$, we have
\begin{align*}
    g(x_0)-\frac{1}{d(x_0)} &\geq g(x_0+ z)-\frac{1}{d(x_0+ z)}.
\end{align*}
Rearranging terms, we get
\begin{align*}
    -\left(g(x_0+ z)-g(x_0)\right)d(x_0)d(x_0+ z)&\geq d(x_0+ z)-d(x_0).
\end{align*}
Using the smoothness of $g$ and the Lipschitzness of $d$, the left-hand side is
\begin{align*}
    -\la \nabla g(x_0), z\ra \big(d(x_0)\big)^2 + o(|z|).
\end{align*}
Comparing this with the definition of $D^+ d(x_0)$, we can see that $-\big(d(x_0)\big)^2\nabla g(x_0)\in D^+ d(x_0)$ as desired.
\end{proof}
\begin{proof}[Proof of Proposition~\ref{p.drop_bdy}]
We only prove the proposition for $f$ assumed to be a subsolution. The case for supersolutions is similar.

Let $\phi:(0,\infty)\times\C$ be a smooth function and assume that $f-\phi$ has a local maximum at $(\underline t, \underline x)\in (0,\infty)\times \C$. If $\underline x\in \mathring\C$, there is nothing to prove. Hence, we assume $\underline x \in \partial \C$. For $r\in(0,\underline t)$, we set $A(r) = [\underline t - r,\underline t+r] \times (\C\cap B(\underline x,r))$.

We can further assume that the maximum is strict, namely,
\begin{align}\label{e.bar_t,x_max}
    f(t,x)-\phi(t,x)<f(\underline t,\underline x) - \phi(\underline t,\underline x),\quad\forall (t,x) \in A(r_0)\setminus \{(\underline t, \underline x)\}
\end{align}
for some sufficiently small $r_0$.

\textit{Step~1.}
Let $d=d_\C$ be given in \eqref{e.d}. For each $\eps>0$, we set
\begin{align}\label{e.Psi_eps}
    \Psi_\eps(t,x) = f(t,x) - \phi(t,x) - \frac{\eps}{d(x)},\quad\forall (t,x)\in \R_+\times \C.
\end{align}
We want to show that $\Psi_\eps$ has a local maximum at $(s_\eps, y_\eps)\in (0,\infty)\times \mathring\C$. 
Since $\Psi_\eps$ is upper semi-continuous with values in $\R\cup \{-\infty\}$, there is $(s_\eps,y_\eps)\in A(r_0)$ such that
\begin{align}\label{e.s_eps,y_eps_max}
    \Psi_\eps(s_\eps,y_\eps) = \sup_{A(r_0)} \Psi_\eps.
\end{align}
Due to the presence of $\frac{\eps}{d(\cdot)}$, we must have $y_\eps \in \mathring \C$.
Setting
\begin{align*}
    O = (\underline t - r_0, \underline t+r_0)\times (\C \cap \itr B(\underline x , r_0)),
\end{align*}
we want to show that $(s_\eps,y_\eps) \in O$ for sufficiently small $\eps$. For every $r\in(0,r_0)$, the continuity of $f-\phi$ and \eqref{e.bar_t,x_max} allow us to find $(t_r,x_r) \in A(r)\cap ((0,\infty)\times \mathring \C)$ such that
\begin{align*}
    f(s,y)-\phi(s,y)<f(t_r,x_r) - \phi(t_r,x_r),\quad\forall (s,y)\in A(r_0)\setminus A(r).
\end{align*}
Hence, for each $r\in(0,r_0)$, there is $\eps_r>0$ such that
\begin{align*}
    \Psi_\eps(s,y) < \Psi_\eps(t_r,x_r), \quad\forall (s,y)\in A(r_0)\setminus A(r),\ \forall \eps <\eps_r,
\end{align*}
which implies that
\begin{align}\label{e.(s_eps,y_eps)}
    (s_\eps, y_\eps)\in  A(r) \subset O,\quad\forall \eps <\eps_r.
\end{align}
Moreover, first taking any subsequential limit of $\eps \to 0$, and then sending $r\to0$, we deduce from \eqref{e.(s_eps,y_eps)} that
\begin{align}\label{e.lim(s_eps,y_eps)}
    \lim_{\eps\to 0}(s_\eps, y_\eps) = (\underline t, \underline x).
\end{align}

\textit{Step~2.}
For each $\eps>0$, let $\zeta_\eps:\R\times \cH\to\R$ be a smooth function compactly supported on $O$, satisfying
\begin{gather}
    0\leq \zeta_\eps\leq 1,\qquad\zeta_\eps(s_\eps,y_\eps) =1, \label{e.zeta_eps_prop}
    \\
    \supp\zeta_\eps\subset (0,\infty)\times \mathring\C, \label{e.zeta_eps_supp}
\end{gather}
where $\supp\zeta_\eps$ is the compact support of $\zeta_\eps$.
Then, for each $\delta,\theta>0$, we define
\begin{align*}
    \Phi(t,x,s,y) = f(t,x) - \phi(s,y) - \frac{\eps}{d(y)} - \frac{2M}{\theta^2}|(t,x)-(s,y)|^2+ \delta \zeta_\eps(s,y),\quad\forall (t,x,s,y)\in O\times O,
\end{align*}
where $M= \sup_{(t,x)\in\bar O}|f(t,x)|$.
We show that for $\theta$ sufficiently small, there is $(t_0,x_0,s_0,y_0)\in O\times O$ depending on $\theta,\delta, \eps$ such that
\begin{align}\label{e.Phi(t_0,x_0,s_0,y_0)_max}
    \Phi(t_0,x_0,s_0,y_0) = \max_{O\times O}\Phi(t,x,s,y).
\end{align}
Let $\omega_f$ denote the modulus of continuity for $f$ on $\overline O$. Using the definition of $M$, and considering two cases depending on $|(t,x)-(s,y)|\leq \theta$ or otherwise, we can get
\begin{align*}
    \Phi(t,x,s,y)\leq \omega_f(\theta)+f(s,y)-\phi(s,y)-\frac{\eps}{d(y)}+\delta\zeta_\eps(s,y),\quad\forall (t,x,s,y)\in O\times O.
\end{align*}
Using this, \eqref{e.s_eps,y_eps_max}, \eqref{e.zeta_eps_prop}, and the definition of $\Phi$, we get
\begin{align*}
    \Phi(t,x,s,y)\leq
    \begin{cases}
    \Phi(s_\eps,y_\eps, s_\eps,y_\eps)+\omega_f(\theta) -\delta, \quad &\text{if }(t,x,s,y)\in O\times (O\setminus \supp\zeta_\eps),
    \\
    \Phi(s_\eps,y_\eps, s_\eps,y_\eps)+\omega_f(\theta), \quad &\text{if }(t,x,s,y)\in O\times  \supp\zeta_\eps.
    \end{cases}
\end{align*}

Let $((t_n,x_n,s_n,y_n))_{n\in\N}$ be a sequence maximizing $\Phi$. Henceforth, let $\theta$ be sufficiently small so that $\omega_f(\theta)<\delta$. Then, the above display implies that, for sufficiently large $n$,
\begin{align}\label{e.(s_n,y_n)_in_supp}
    (s_n,y_n)\in \supp\zeta_\eps.
\end{align}
Note that if $|(t_n,x_n) - (s_n,y_n)|>\theta$, then we can verify using the definition of $M$ that
\begin{align*}
    \Phi(t_n,x_n,s_n,y_n)\leq f(t_n,x_n) - \phi(s_n,y_n) - \frac{\eps}{d(y_n)} - 2M+ \delta \zeta_\eps(s_n,y_n)\leq \Phi(s_n,y_n,s_n,y_n).
\end{align*}
Hence, by replacing $(t_n,x_n)$ by $(s_n,y_n)$ if necessary, we may assume that, for sufficiently large $n$,
\begin{align*}
    |(t_n,x_n)-(s_n,y_n)|\leq \theta.
\end{align*}
This along with \eqref{e.(s_n,y_n)_in_supp} allow us to pass to the limit in $n$ to deduce the existence of a maximizer in \eqref{e.Phi(t_0,x_0,s_0,y_0)_max}, which satisfies
\begin{align}\label{e.(t_0,x_0)-(s_0,y_0)}
    |(t_0,x_0)-(s_0,y_0)|\leq \theta.
\end{align}
Since $\supp\zeta_\eps\subset O$ is compact, for sufficiently small $\theta$, we can ensure by \eqref{e.(t_0,x_0)-(s_0,y_0)} that $(t_0,x_0,s_0,y_0)\in O\times O$. Using \eqref{e.zeta_eps_supp}, \eqref{e.(s_n,y_n)_in_supp} (taking $n\to\infty$ therein) and \eqref{e.(t_0,x_0)-(s_0,y_0)}, we have that $(t_0,x_0)\in (0,\infty)\times \mathring\C$ and $(s_0,y_0)\in (0,\infty)\times \mathring\C$ for sufficiently small $\theta$.

\textit{Step~3.}
Since the function $(t,x)\mapsto\Phi(t,x,s_0,y_0)$ achieves its maximum over $O$ at $(t_0,x_0)\in (0,\infty)\times \mathring\C$, the assumption that $f$ is a subsolution to $\HJ(\mathring\C,\sF)$ implies that
\begin{align}\label{e.conseq_subsol}
    \frac{4M}{\theta^2}(t_0-s_0) -\sF\left(\frac{4M}{\theta^2}(x_0-y_0)\right)\leq 0.
\end{align}
On the other hand, since the function $(s,y)\mapsto\Phi(t_0,x_0,s,y)$ achieves its maximum at an interior point $(s_0,y_0)$, we can compute
\begin{gather*}
    \partial_t\phi(s_0,y_0)-\frac{4M}{\theta^2}(t_0-s_0)-\delta\partial_t\zeta_\eps(s_0,y_0)=0,
    \\
    \frac{1}{\eps}\left(d(y_0)\right)^2\left(\nabla \phi(s_0,y_0)-\frac{4M}{\theta^2}(x_0-y_0)-\delta \nabla \zeta_\eps(s_0,y_0)\right)\in D^+d(y_0), \notag
\end{gather*}
where the second relation follows from Lemma~\ref{l.d}~\eqref{i.g-1/d}.
By Lemma~\ref{l.d}~\eqref{i.Dd}, we obtain from the second relation that
\begin{align*}
    \frac{4M}{\theta^2}(x_0-y_0) =\nabla \phi(s_0,y_0)-\delta \nabla \zeta_\eps(s_0,y_0)-p
\end{align*}
for some $p\in \C^*$. Using this, \eqref{e.conseq_subsol} and the assumption that $\sF$ is $\C^*$-increasing, we obtain
\begin{align}\label{e.d_tphi-H...}
    \partial_t\phi(s_0,y_0)-\delta\partial_t\zeta_\eps(s_0,y_0)-\sF\left(\nabla \phi(s_0,y_0)-\delta \nabla \zeta_\eps(s_0,y_0)\right)\leq 0.
\end{align}

\textit{Step~4.}
Recall that $(s_0,y_0)$ depends on $\theta,\delta,\eps$ and $\zeta_\eps$ only depends on $\eps$. We claim that
\begin{align}\label{e.limlimlim}
    \lim_{\eps\to0}\lim_{\delta\to0}\lim_{\theta\to0}(s_0,y_0)=(\underline t,\underline x),
\end{align}
along some subsequence.
Assuming this, and passing \eqref{e.d_tphi-H...} to the limits in the same order as in \eqref{e.limlimlim}, we arrive at
\begin{align*}
    \partial_t\phi(\underline t,\underline x) - \sF\left(\nabla\phi\left(\underline t,\underline x\right)\right)\leq 0,
\end{align*}
which completes our proof that $f$ is a subsolution on $(0,\infty)\times \C$.

It remains to show \eqref{e.limlimlim}. Due to \eqref{e.(t_0,x_0)-(s_0,y_0)}, we have
\begin{align*}
    \lim_{\theta\to0}(t_0,x_0) = \lim_{\theta\to0}(s_0,y_0) = (s_1,y_1)\in \supp\zeta_\eps
\end{align*}
for some $(s_1,y_1)$ depending on $\delta,\eps$. Using \eqref{e.Phi(t_0,x_0,s_0,y_0)_max}, we have $\Phi(t_0,x_0,s_0,y_0)\geq \sup_{(t,x)\in O} \Phi(t,x,t,x)$ and thus
\begin{align*}
    f(t_0,x_0)-\phi(s_0,y_0)-\frac{\eps}{d(y_0)}+\delta\zeta_\eps(s_0,y_0)\geq f(t,x)-\phi(t,x)-\frac{\eps}{d(x)}+\delta\zeta_\eps(t,x),\quad\forall (t,x)\in O.
\end{align*}
Let $(s_2,y_2)\in\supp\zeta_\eps$ be a subsequential limit of $(s_1,y_1)$ as $\delta\to 0$. So, $(s_2,y_2)$ only depends on $\eps$. Hence, sending $\theta\to0$ and then $\delta\to0$, we obtain from the above display that
\begin{align*}
    f(s_2,y_2)-\phi(s_2,y_2)-\frac{\eps}{d(y_2)}\geq f(t,x)-\phi(t,x)-\frac{\eps}{d(x)},\quad\forall (t,x)\in O.
\end{align*}
Hence, $(s_2,y_2)$ maximizes $\Psi_\eps$ introduced in \eqref{e.Psi_eps} over $O$. By the same argument used to derive \eqref{e.lim(s_eps,y_eps)}, we have that $\lim_{\eps\to0}(s_2,y_2)= (\underline t, \underline x)$.
\end{proof}
\subsection{Modifications of the nonlinearity}

Recall that $\C$ is assumed to satisfy \ref{i.assump_C} throughout.
We need the following result for the construction of modifications of the nonlinearity and proofs involving the trick of doubling variables.
\begin{lemma}\label{l.cone_itr}
Under \ref{i.assump_C}, $\itr(\C\cap\C^*)\neq \emptyset$.
\end{lemma}
\begin{proof}
It is easy to see that in finite dimensions, a nonempty closed convex cone is pointed if and only if its dual cone has a nonempty interior (equivalently, full-dimensional). Indeed, if the cone is not pointed, then it contains a line through the origin which must be orthogonal to its dual cone, implying that its dual cone is not full-dimensional. The converse is straightforward. An immediate consequence of this is that, under \ref{i.assump_C}, $\C^*$ must be pointed and has a nonempty interior.

By \cite[Proposition~6.26]{bauschke2011convex} (stated for the \textit{polar cone} which differs from the dual cone by a minus sign), we have $(\C\cap\C^*)^* = \C^* + (\C^*)^*= \C^*+\C$. Hence, to show $\itr(\C\cap\C^*)\neq \emptyset$, it suffices to show that $\C^*+\C$ is pointed. Suppose otherwise. Then, there are $a,b\in\C$ and $x^*,y^*\in \C^*$ such that $a+x^* = -(b+y^*)$. Rearranging and multiplying both sides by $a+b$, we get $0\leq |a+b|^2 = -\la a+b, x^*+y^*\ra \leq 0$. So, we must have $a = -b$, and, in a similar way, $x^*=-y^*$. Since $\C$ and $\C^*$ are pointed, we get $a=b=x^*=y^*=0$, which shows that $\C+\C^*$ is pointed.
\end{proof}

We are ready to construct modifications.

\begin{lemma}[Construction of modifications]\label{l.ext_loclip}
Let $\H:\cH\to\R$ be a function.
\begin{enumerate}
    \item \label{i.ext_lip} If $\H\lfloor_{\C}\in\Gamma^\nearrow_\mathrm{Lip}(\C)$, then there is $\sF\in \Gamma^\nearrow_\mathrm{Lip}(\cH)$ satisfying $\sF\lfloor_{\C} = \H\lfloor_{\C}$ and $\|\sF\|_\mathrm{Lip}= \|\H\lfloor_\C\|_\mathrm{Lip}$.
    \item \label{i.ext_loclip} If 
$\H\lfloor_\C\in \Gamma^\nearrow_\mathrm{locLip}(\C)$, then, for every $R>0$, there is $\sF\in \Gamma^\nearrow_\mathrm{Lip}(\cH)$ satisfying $\sF\lfloor_{\C\cap B(0,R)} = \H\lfloor_{\C\cap B(0,R)}$.
\end{enumerate}
Moreover, in both cases, the construction $\sF$ satisfies the following:
\begin{itemize}
    \item if $\H\lfloor_\C$ is convex, then $\sF$ is convex;
    \item if $\H\lfloor_\C\geq 0$, then $\sF\geq 0$.
\end{itemize}
\end{lemma}

\begin{proof}
Part~\eqref{i.ext_lip}.
We define $\sF(x)=\inf\{\H(y)\,|\,y\in \C\cap(x+\C^*)\}$ for $x\in\cH$. Clearly, $\sF$ is $\C^*$-increasing as $\C\cap(x'+\C^*)\subset \C\cap(x+\C^*)$ if $x'-x\in\C^*$. Since $\H\lfloor_\C$ is $\C^*$-increasing, we can easily check that $\sF\lfloor_\C=\H\lfloor_\C$.
Now, we show that $\sF$ is Lipschitz. Fix any $x_1, x_2\in\cH$. Let $ p $ be the projection of $x_1-x_2$ to $\C$. Then
\begin{align}\label{e.proj}
    \la x_1-x_2- p , c- p \ra\leq 0,\quad\forall c\in\C.
\end{align}
By setting $c=s p \in\C$ for $s\geq 0$ and varying $s$, we get 
\begin{align}\label{e.Lip0}
\la x_1-x_2- p ,  p \ra =0. 
\end{align}
By this, \eqref{e.proj} implies that $x_2-x_1+ p  \in\C^*$. Thus, for any $y_2\in \C\cap(x_2+\C^*)$, we have $y_2+ p \in \C\cap(x_1+\C^*)$. Using $\sF\lfloor_\C=\H\lfloor_\C$ and that $\sF$ is $\C^*$-increasing, we get that
\begin{align}\label{e.>F(x_1)}
    \H(y_2+ p )=\sF(y_2+ p )\geq \sF(x_1).
\end{align}
Also, due to the assumption that $\H\lfloor_\C$ is Lipschitz, setting $L = \|\H\lfloor_\C\|_\mathrm{Lip}$, we have that
\begin{align*}
    |\H(y_2+ p )-\H(y_2)|\leq L| p |.
\end{align*}
Combining the two displays above gives us 
\begin{align}\label{e.Liprho}
\sF(x_1)-\H(y_2)\leq L| p |,\quad\forall y_2\in\C\cap(x_2+\C^*).
\end{align}
Note that, by \eqref{e.Lip0}, we have $|x_1-x_2- p |^2=|x_1-x_2|^2-| p |^2$ and thus
\begin{align}\label{e.x_1-x_2}
    |x_1-x_2|\geq | p |.
\end{align}
Using this and taking the supremum over $y_2\in\C\cap(x_2+\C^*)$ in \eqref{e.Liprho}, we achieve that
\begin{align*}
    \sF(x_1)-\sF(x_2)\leq L |x_1-x_2|,
\end{align*}
proving the Lipschitzness of $\sF$, which also implies $\|\sF\|_\mathrm{Lip}= L$. This completes the proof of \eqref{i.ext_lip}.

Let us verify additional claims.
By the construction of $\sF$, it clearly satisfies that and if $\H\geq 0$, then $\sF\geq 0$. We now assume that $\H\lfloor_\C$ is convex.
Fix any $x,y\in\cH$, and let $x'\in \C\cap (x+\C^*), y'\in \C\cap(y+\C^*)$. Then, for any $\lambda\in [0,1]$, we have $\lambda x'+(1-\lambda)y' \in \C\cap (\lambda x+(1-\lambda)y+\C^*)$. Due to the convexity of $\H\lfloor_\C$, 
\begin{align*}
    \sF(\lambda x+(1-\lambda)y) \leq \H(\lambda x'+(1-\lambda)y') \leq \lambda \H (x') + (1-\lambda)\H(y').
\end{align*}
Optimizing over $x',y'$, we can deduce that $\sF$ is convex.

Part~\eqref{i.ext_loclip}. Fix any $v\in \itr(\C\cap\C^*)$ which is allowed by Lemma~\ref{l.cone_itr}. Then, there is $\delta>0$ such that $v+B(0,\delta)\subset \C^*$. Hence, for each $x\in\C$, we have $\la v - \delta y, x\ra \geq 0$ for all $y\in B(0,1)$ implying that $\la v, x\ra \geq \delta |x|$ by choosing $y=\frac{x}{|x|}$. On the other hand, the Cauchy-Schwarz inequality gives $\la v,x\ra \leq |v||x|$. Therefore, there is a constant $C>0$ such that
\begin{align}\label{e.<v,x>}
    C^{-1}|x|\leq \la v, x\ra \leq C|x|,\quad\forall x \in \C.
\end{align}

For $l>0$, we set
\begin{align*}
    A(l) = \{x\in\C:\la v,x\ra \leq l\}.
\end{align*}
For every $R>0$, we choose $r$ sufficiently large to ensure that $A(2r)\supset \C\cap B(0,R)$ and
\begin{align}\label{e.r>R}
    (2C^2+1)R<2Cr.
\end{align}
We also set $L = \|\H\lfloor_{A(2r)}\|_\mathrm{Lip}$. We define
\begin{align*}
    \tilde \H =
    \begin{cases}
    \H(x)\vee\left(\H(0)+2LC(\la v, x\ra -r)\right),\quad & x\in A(2r),
    \\
    \H(0)+2LC(\la v, x\ra -r)\quad & x\in \C\setminus A(2r),
    \end{cases}
\end{align*}
For $x\in\C$ satisfying $\la v, x\ra =2r$, by \eqref{e.<v,x>}, we have that
\begin{align*}
    \H(0) + 2LC(\la v, x\ra -r) =\H(0)+LC\la v, x\ra \geq \H(0)+L|x|\geq \H(x).
\end{align*}
Hence, $\tilde\H$ is continuous at points on $\{x\in\C:\la v,x\ra=2r\}$. Using this, $v\in\C$, and $\H\lfloor_\C\in \Gamma^\nearrow_\mathrm{locLip}(\C)$, we can verify that $\tilde\H\lfloor_\C\in \Gamma^\nearrow_\mathrm{Lip}(\C)$.

Using \eqref{e.r>R}, we have
\begin{align*}
    2LC(C|x|-r)\leq -L|x|, \quad\forall x\in B(0,R),
\end{align*}
which along with \eqref{e.<v,x>} implies that, for $x\in \C\cap B(0,R)$,
\begin{align*}
    \H(0) + 2LC(\la v, x\ra -r)&\leq \H(0) + 2LC(C|x| -r)\\
    &\leq \H(0)-L|x|\\
    &\leq \H(x).
\end{align*}
Hence, recalling that $\C\cap B(0,R)\subset A(2r)$, we obtain that $\tilde\H\lfloor_{\C\cap B(0,R)} = \H\lfloor_{\C\cap B(0,R)}$.

We set $\sF(x)=\inf\{\tilde\H(y)\,|\,y\in \C\cap(x+\C^*)\}$. Then, $\sF$ satisfies the desired conditions, as a result of the first part applied to $\sF$. 

Then, we verify additional conditions.
By construction, it is clear that if $\H\lfloor_\C\geq 0$, then $\tilde \H\geq 0$; and if $\H\lfloor_\C$ is convex, then $\tilde\H$ is convex. Thus, by the arguments in the first part, we know that if $\H\lfloor_\C\geq 0$, then $\sF\geq 0$; and if $\H\lfloor_\C$ is convex, then $\sF$ is convex.
\end{proof}

\begin{lemma}[Modification for a growing nonlinearity]\label{l.mod_grow}
If $\H:\cH\to\R$ satisfies $\H\lfloor_\C\in \Gamma^\nearrow_\mathrm{locLip}(\C)$ and
\begin{align}\label{e.liminfH}
    \liminf_{\substack{x\to\infty \\ x\in \C}}\H(x) =\infty,
\end{align}
then, there is $\sF\in\Gamma^\nearrow_\mathrm{locLip}(\cH)$ satisfying $\sF\lfloor_\C=\H\lfloor_\C$.
\end{lemma}
\begin{proof}
We define
\begin{align*}
    \sF(x) = \inf\{\H(y)\,|\,y\in \C\cap(x+\C^*)\},\quad\forall x \in \cH.
\end{align*}
We first prove the following claim: \textit{for every $R>0$, there is $r>0$ such that, for all $x\in B(0,R)$, there is $y\in \C\cap (x+\C^*)\cap B(0,r)$ satisfying $\sF(x) = \H(y)$.}
Fix any $u\in \itr(\C\cap\C^*)$ which is allowed by Lemma~\ref{l.cone_itr}. Then, there is $\delta>0$ such that $u +  B(0,\delta R) \subset \C^*$. Setting $v =\delta^{-1} u$, we have $v - x \in\C^*$ for all $x\in B(0,R)$. By the definition of $\sF$, we have that $\sF(x)\leq \H(v)$ for all $x\in B(0,R)$. By the assumption \eqref{e.liminfH}, there is $r>0$ such that $\inf_{\C\setminus B(0,r)} \H >\H(v)$. Hence, the infimum in $\sF(x)$ must be achieved at some $y\in \C\cap B(0,r)$, proving the claim.

By the same argument in the proof of Lemma~\ref{l.ext_loclip}~\eqref{i.ext_lip}, $\sF$ is $\C^*$-increasing and satisfies $\sF\lfloor_\C=\H\lfloor_\C$. Now, we show that $\sF$ is locally Lipschitz. Fix any $R>0$ and let $r>0$ be given by the claim proved above. Let $x_1,x_2\in \C\cap B(0,R)$ and $p$ be the projection of $x_1-x_2$ to $\C$. Proceed as in the proof of Lemma~\ref{l.ext_loclip}~\eqref{i.ext_lip}, we arrive at (see \eqref{e.>F(x_1)})
\begin{align*}
    \H(y_2+p)=\sF(y_2+p)\geq \sF(x_1)
\end{align*}
where, allowed by the above claim, we can choose $y_2\in \C\cap (x_2+\C^*)$ to satisfy $|y_2|\leq r$ and $\sF(x_2) = \H(y_2)$. Using \eqref{e.x_1-x_2}, we have $y_2,\, y_2+p  \in \C\cap B(0,2R+r)$. Setting $L$ to be Lipschitz coefficient of $\H\lfloor_{\C\cap B(0,2R+r)}$. This along with the above display, the fact that $\sF(x_2) = \H(y_2)$, and \eqref{e.x_1-x_2}, implies that
\begin{align*}
    \sF(x_1)\leq \H(y_2) + L|p|\leq \sF(x_2)+ L|x_1-x_2|.
\end{align*}
Therefore, $\sF$ is Lipschitz on $B(0,R)$, completing the proof.
\end{proof}

\begin{lemma}[Invariance of solutions under modification]\label{l.equiv}
Let $\sF:\cH\to\R$, $\H:\cH\to\R$ and $f:\R_+\times \C\to\R$. Assume one of the following:
\begin{enumerate}
    \item $\sF\lfloor_\C = \H\lfloor_\C$, and $f\in \mathfrak{M}$.
    
    \item $\sF\lfloor_{\C\cap B(0,R)} = \H\lfloor_{\C\cap B(0,R)}$ for some $R>0$, and $f\in\mathfrak{M}$ satisfies $\sup_{t\in\R_+}\|f(t,\cdot)\|_\mathrm{Lip}\leq R$. 
    
\end{enumerate}
Then, $f$ is a subsolution (resp., supersolution) of $\HJ(\mathring\C,\sF)$ if and only if $f$ is a subsolution (resp., supersolution) of $\HJ(\mathring\C,\H)$.

\end{lemma}

\begin{proof}
Let us verify the equivalence for subsolutions. The argument for supersolutions is the same. Suppose that $f-\phi$ achieves a local maximum at $(t,x)\in(0,\infty)\times\mathring\C$ for a smooth function $\phi$. It suffices to show that
\begin{align}\label{e.F(.)=H(.)}
    \sF(\nabla\phi(t,x)) = \H(\nabla \phi(t,x)).
\end{align}

Let us verify \eqref{e.F(.)=H(.)} under the first assumption. For any $y\in\C^*$ and $\epsilon>0$ small enough, we have
\begin{align*}
    \phi(t,x+\epsilon y)-\phi(t,x)\geq  f(t,x+\epsilon y)-f(t,x)\geq 0,
\end{align*}
where in the last inequality we used $f(t,\cdot)\in\Gamma^\nearrow(\C)$ due to $f\in\mathfrak{M}$. Sending $\epsilon$ to $0$, we obtain that
\begin{align*}
    \la \nabla \phi(t,x), y\ra\geq 0,\quad\forall y \in \C^*,
\end{align*}
implying that $\nabla \phi(t,x)\in(\C^*)^*=\C$. Hence, \eqref{e.F(.)=H(.)} follows from the condition that $\sF\lfloor_\C = \H\lfloor_\C$.

Under the second assumption, we also have $\nabla\phi(t,x)\in\C$. Now, since $\|f(t,\cdot)\|_\mathrm{Lip}\leq R$,  we have that, for every $z\in\cH$ and sufficiently small $\delta>0$,
\begin{align*}
    \phi(t,x-\delta z)-\phi(t,x)\geq  f(t,x-\delta z)-f(t,x)\geq-R\delta|z|,
\end{align*}
which implies $\la \nabla\phi(t,x),z\ra\leq R|z|$ for every $z\in\cH$ and thus $\nabla\phi(t,x)\in B(0,R)$. Hence, \eqref{e.F(.)=H(.)} follows from $\sF\lfloor_{\C\cap B(0,R)} = \H\lfloor_{\C\cap B(0,R)}$.
\end{proof}

\section{Well-posedness}
\label{s.well-posed}
Recall that, in this work, we interpret the well-posedness as the comparison principle together with the existence of solutions.

We first prove the comparison principle in Proposition~\ref{p.cp_H} for $\HJ(\mathring\C,\sF)$, which immediately yields the comparison principle, Corollary~\ref{c.cp_interior}, for $\HJ(\C,\sF)$. Then, we use results from Section~\ref{s.prelim} to obtain the comparison principle, Corollary~\ref{c.cp_C}, for $\HJ(\mathring\C,\H)$. Finally, using Perron's method in Proposition~\ref{p.exist_sol_H}, we show the existence of solutions of $\HJ(\mathring\C,\sF;\psi)$ for Lipschitz $\psi,\sF$.

We emphasize that Proposition~\ref{p.exist_sol_H} alone is not enough to prove the existence of solutions in Theorem~\ref{t}~\eqref{i.main_exist}. With results in this section and Section~\ref{s.prelim}, we are only able to show the well-posedness of $\HJ(\C,\sF)$ (and $\HJ(\mathring\C,\sF)$) for $\sF\in\Gamma^\nearrow_\mathrm{Lip}(\cH)$, as commented in Remark~\ref{r.well-posed}. For Theorem~\ref{t}~\eqref{i.main_exist}, we need more properties on the solution to be proved in ensuing sections (see the proof of Theorem~\ref{t}~\eqref{i.main_exist} in Section~\ref{s.intro_main_result}).

\subsection{Comparison principles}

\begin{proposition}[Comparison principle]\label{p.cp_H}
Assume one of the following:
\begin{enumerate}[label=\rm{(\roman*)}]
    \item \label{i.case1} $\sF\in\Gamma^\nearrow_\mathrm{Lip}(\cH)$, and $u,v\in\mathfrak{L}$;

    \item \label{i.case2} $\sF\in\Gamma^\nearrow_\mathrm{locLip}(\cH)$, and $u,v\in\mathfrak{L}_\mathrm{Lip}$;
\end{enumerate}
If $u,v$ are a subsolution and a supersolution of $\HJ(\mathring\C,\sF)$, respectively, then $\sup_{\R_+\times\C}(u-v)=\sup_{\{0\}\times \C}(u-v)$.
\end{proposition}

\begin{proof}
We treat both cases simultaneously.
Setting $C_0 = \sup_{\{0\}\times \C}(u-v)$, we can assume that $C_0$ is finite, otherwise there is nothing to show.
We argue by contradiction and assume that $\sup_{\R_+\times\C}(u-v)>\sup_{\{0\}\times\C}(u-v)$. Then, we can find $T>0$ sufficiently large so that
\begin{align}\label{e.u-v}
    \sup_{[0,T)\times\C}(u-v)>\sup_{\{0\}\times\C}(u-v).
\end{align}

\textit{Step~1.}
We fix some constants and introduce auxiliary functions.
Due to $u,v\in\mathfrak{L}$ in both cases, we can fix a constant $L>1$ to satisfy
\begin{gather*}
    L>\|v(0,\cdot)\|_\mathrm{Lip},
    \\
    |u(t,x)-u(0,x)|\vee |v(t,x)-v(0,x)|\leq Lt, \quad\forall (t,x)\in\R_+\times\C.
\end{gather*}
If necessary, we make $L$ larger to ensure
\begin{align}\label{e.L_case_ii}
    L > \sup_{t\in\R_+}\|v(t,\cdot)\|_\mathrm{Lip},\quad\text{in case~\ref{i.case2}}.
\end{align}
We set
\begin{align}\label{e.V=}
    V=
    \begin{cases}
    \|\sF\|_\mathrm{Lip},\quad &\text{in case~\ref{i.case1}},
    \\
    \|\sF\lfloor_{B(0,2L+2)}\|_\mathrm{Lip}, &\text{in case~\ref{i.case2}},
    \end{cases}
\end{align}
Let $\theta:\R\to\R_+$ be a smooth function satisfying
\begin{align}\label{e.theta}
    |\theta'(r)|\leq 1, \qquad\text{and}\qquad r\vee0 \leq \theta(r)\leq (r+1)\vee0,\quad\forall r\in\R,
\end{align}
where $\theta'$ is the derivative of $\theta$.
For $R>0$ to be chosen, we define
\begin{align*}
    \chi(t,x) = \theta\left(\left(1+|x|^2\right)^\frac{1}{2}+Vt-R\right),\quad\forall (t,x)\in \R_+\times \C.
\end{align*}
It is immediate that
\begin{gather}
    \sup_{(t,x)\in\R_+\times\C}|\nabla\chi(t,x)|\leq 1,\label{e.cp2.|nabla_Phi|}
    \\
    \partial_t\chi\geq V|\nabla \chi|,\label{e.cp2.d_tPhi>}
    \\
    \chi(t,x)\geq |x|-R, \quad\forall (t,x)\in\R_+\times\C. \label{e.cp2.Phi(t,x)>M|x|}
\end{gather}
For $\delta\in(0,2V)$ to be determined, we set
\begin{align}\label{e.eps=delta/2V}
    \eps = \frac{\delta}{2V}\in(0,1)
\end{align}
and define
\begin{align*}
    \zeta_1(t,t',x)=\chi(t,x) + \delta t+\frac{1}{T-t}+\frac{1}{T-t'},\quad \forall (t,t',x) \in [0,T)^2\times\C.
\end{align*}
Let $d=d_\C$ in \eqref{e.d}.
For each $\alpha>1$, we introduce
\begin{align*}
    \zeta_2(x,y) = \frac{\alpha}{2}|x-y|^2 + \frac{\delta}{d(y)}+\eps|y|,\quad\forall (x,y)\in \C^2.
\end{align*}
In view of \eqref{e.u-v}, we fix $\delta>0$ sufficiently small and $R,T$ sufficiently large so that
\begin{gather}
    \sup_{(t,x)\in[0,T)\times\C}(u(t,x)-v(t,x)-\zeta_1(t,t,x)-\zeta_2(x,x))\notag
    \\
    >\sup_{x\in\C}(u(0,x)-v(0,x)-\zeta_1(0,0,x)-\zeta_2(x,x)), \label{e.cp2.u-v-chi}
\end{gather}
Note that $\zeta_2(x,x)$ does not dependent on $\alpha$.
We introduce
\begin{align*}
    \Psi_\alpha(t,t',x,x',y)= u(t,x)-v(t',x')-\frac{\alpha}{2}(|t-t'|^2+|x-x'|^2)-\zeta_1(t,t',x)-\zeta_2(x,y),
    \\
    \quad\forall (t,t',x,x',y)\in[0,T)^2\times\C^3. 
\end{align*}
By the semi-continuity of $u$ and $v$, we have that $\Psi_\alpha$ is upper semicontinuous.
Throughout, we fix any $y_0\in \mathring\C$.

\textit{Step~2.}
We show the existence of a maximizer of $\Psi_\alpha$ and derive estimates on this maximizer.
By the definition of $L$ and $C_0$, we can see that, for all $t,t'\in[0,T)$ and $x,x'\in\C$,
\begin{align}\label{e.u-v<2LT...}
    u(t,x)-v(t',x')\leq 2LT+u(0,x)-v(0,x)+L|x-x'|\leq 2LT+C_0 + L|x-x'|,
\end{align}
which along with $\alpha>1$ and \eqref{e.cp2.Phi(t,x)>M|x|} implies that
\begin{align*}
    \Psi_\alpha(t,t',x,x',y)\leq 2LT +C_0+L|x-x'|-\frac{1}{2}|x-x'|^2-(|x|-R)-\frac{1}{T-t}-\frac{1}{T-t'}-\eps|y|.
\end{align*}
Hence, $\Psi_\alpha$ is bounded from above uniformly in $\alpha>1$ and decays as $|x|,|x'|,|y|\to\infty$ or $t,t'\to T$, which implies that $\Psi_\alpha$ achieves its supremum at some $(t_\alpha,t'_\alpha,x_\alpha,x'_\alpha,y_\alpha)$.
Setting $C_1=\Psi_\alpha(0,0,y_0,y_0,y_0)$ which is independent of $\alpha$, we have $C_1\leq \Psi_\alpha(t_\alpha,t'_\alpha,x_\alpha,x'_\alpha,y_\alpha)$ which along with the above display implies that there is a constant $C_2>0$ such that
\begin{align}\label{e.xx'yy'<C_1}
    |x_\alpha|,\ |x'_\alpha|,\ |y_\alpha|\leq C_2,\quad\forall \alpha>1.
\end{align}
Since we also have
\begin{align*}
    C_1\leq \Psi_\alpha(t_\alpha,t'_\alpha,x_\alpha,x'_\alpha,y_\alpha)\leq 2LT+C_0+2LC_2-\frac{\alpha}{2}(|t_\alpha-t'_\alpha|^2+|x_\alpha-x'_\alpha|^2),
\end{align*}
we can see that, as $\alpha\to\infty$,
\begin{align}\label{e.|t-t'|,|x-x'|}
    |t_\alpha-t'_\alpha|,\ |x_\alpha-x'_\alpha| = O(\alpha^{-\frac{1}{2}}).
\end{align}
In case~\ref{i.case2}, we can obtain a better bound on $|x_\alpha-x'_\alpha|$.
Using
\begin{align*}
    0\geq \Psi_\alpha(t_\alpha,t'_\alpha,x_\alpha,x_\alpha,y_\alpha)-\Psi_\alpha(t_\alpha,t'_\alpha,x_\alpha,x'_\alpha,y_\alpha) = v(t'_\alpha,x'_\alpha)-v(t'_\alpha,x_\alpha)+\frac{\alpha}{2}|x_\alpha-x'_\alpha|^2,
\end{align*}
and the property of $L$ in \eqref{e.L_case_ii}, we get
\begin{align}\label{e.cp2.|x_alpha-x'_alpha|}
    \alpha|x_\alpha-x'_\alpha|\leq  2L,\quad\text{in case~\ref{i.case2}}.
\end{align}

Using $\Psi_\alpha(t_\alpha,t'_\alpha,x_\alpha,x'_\alpha,y_\alpha)\geq C_1$ and \eqref{e.u-v<2LT...}, we have
\begin{align*}
    2LT+C_0 + L|x_\alpha-x'_\alpha|-\frac{\delta}{d(y_\alpha)} \geq C_1,
\end{align*}
which along with \eqref{e.xx'yy'<C_1} implies that there is a constant $C_3>0$ such that
\begin{align}\label{e.d(y),d(y')>C_3}
    d(y_\alpha)>C_3,\quad\forall \alpha>1.
\end{align}
Hence, $y_\alpha\in\mathring\C$ for all $\alpha$. Since $y\mapsto \Psi_\alpha(t_\alpha,t'_\alpha,x_\alpha,x'_\alpha,y)$ has a local maximum at $y_\alpha$, by Lemma~\ref{l.d}~\eqref{i.g-1/d}, we have that
\begin{align*}
    \frac{1}{\delta}((d(y_\alpha))^2\left(\alpha(y_\alpha-x_\alpha)+\eps\frac{ y_\alpha}{|y_\alpha|}\right)\in D^+d(y_\alpha).
\end{align*}
Using \eqref{e.d(y),d(y')>C_3} and Lemma~\ref{l.d}~\eqref{i.Dd}, we have that
\begin{gather}
    |x_\alpha-y_\alpha|=O(\alpha^{-1}) \quad\text{as $\alpha\to\infty$}, \label{e.|x-y|,|x'-y'|}
    \\
    \alpha (y_\alpha-x_\alpha)+p\in\C^*\quad\text{for some $p$ satisfying $|p|=\eps $}. \label{e.y-x_in_C^*}
\end{gather}

\textit{Step~3.}
We want to fix an appropriate value of $\alpha$.
In view of \eqref{e.xx'yy'<C_1}, \eqref{e.|t-t'|,|x-x'|}, and \eqref{e.|x-y|,|x'-y'|}, passing to a subsequence if necessary, we may assume $t_\alpha,t'_\alpha\to t_0$ and $x_\alpha,x'_\alpha,y_\alpha\to x_0$ for some $(t_0,x_0)\in [0,T]\times \C$ as $\alpha\to\infty$.
Then, we show $t_0\in(0,T)$.
Since \eqref{e.u-v<2LT...}, \eqref{e.xx'yy'<C_1}, and the definition of $C_1$ imply that
\begin{align*}
    C_1\leq \Psi_\alpha(t_\alpha,t'_\alpha,x_\alpha,x'_\alpha,y_\alpha)\leq 2LT+C_0+2LC_2-\frac{1}{T-t_\alpha},
\end{align*}
we must have that $t_\alpha$ is bounded away from $T$ uniformly in $\alpha$, which implies $t_0<T$. Since
\begin{gather*}
    u(t_\alpha,x_\alpha) - v(t'_\alpha,x'_\alpha)-\zeta_1(t_\alpha,t'_\alpha,x_\alpha)-\zeta_2(x_\alpha,y_\alpha)\geq \Psi_\alpha(t_\alpha,t'_\alpha,x_\alpha,x'_\alpha,y_\alpha) \\
    \geq \sup_{(t,x)\in[0,T)\times\C}\Psi_\alpha(t,t,x,x,x)= \sup_{(t,x)\in[0,T)\times\C}(u(t,x)-v(t,x)-\zeta_1(t,t,x)-\zeta_2(x,x)),
\end{gather*}
sending $\alpha\to\infty$, we get
\begin{gather*}
    u(t_0,x_0)-v(t_0,x_0)-\zeta_1 (t_0,t_0,x_0) - \zeta_2(x_0,x_0)
    \\
    \geq \sup_{(t,x)\in[0,T)\times\C}(u(t,x)-v(t,x)-\zeta_1(t,t,x)-\zeta_2(x,x)).
\end{gather*}
This along with \eqref{e.cp2.u-v-chi} implies that $t_0>0$. In conclusion, we have $t_0\in(0,T)$, and thus $t_\alpha,t'_\alpha\in(0,T)$ for sufficiently large $\alpha$. In addition, due to \eqref{e.|t-t'|,|x-x'|}, \eqref{e.|x-y|,|x'-y'|}, and \eqref{e.d(y),d(y')>C_3}, we can see that $x_\alpha,x'_\alpha\in\mathring\C$ for sufficiently large $\alpha$. Henceforth, we fix any such $\alpha$.

\textit{Step~4.}
We conclude the proof. Since the function
\begin{align*}
    (t,x)\mapsto \Psi_\alpha(t,t'_\alpha,x,x'_\alpha,y_\alpha)
\end{align*}
achieves its maximum at $(t_\alpha,x_\alpha)\in (0,T)\times \mathring\C$, by the assumption that $u$ is a subsolution, we have
\begin{align}\label{e.cp2.u_sub_Phi_alpha}
    \alpha(t_\alpha-t'_\alpha) + \delta +(T-t_\alpha)^{-2}+\partial_t\chi(t_\alpha,x_\alpha)-\sF\left(\alpha(x_\alpha - x'_\alpha)+\nabla\chi(t_\alpha,x_\alpha)-\alpha(y_\alpha-x_\alpha)\right)\leq 0
\end{align}
On the other hand, since the function
\begin{align*}
    (t',x')\mapsto \Psi_\alpha(t_\alpha,t',x_\alpha,x',y_\alpha)
\end{align*}
achieves its minimum at $(t'_\alpha,x'_\alpha)\in (0,T)\times \mathring\C$, by the assumption that $v$ is a supersolution, we have
\begin{align}\label{e.cp2.v_super_Phi_alpha}
    \alpha(t_\alpha-t'_\alpha) -(T-t'_\alpha)^{-2} -\sF\left(\alpha(x_\alpha - x'_\alpha)\right)\geq 0.
\end{align}
Taking the difference of \eqref{e.cp2.u_sub_Phi_alpha} and \eqref{e.cp2.v_super_Phi_alpha} and using the positivity of $(T-t_\alpha)^{-2}$ and $(T-t'_\alpha)^{-2}$, we have
\begin{align*}
    \delta \leq \sF\left(\alpha(x_\alpha - x'_\alpha)+\nabla\chi(t_\alpha,x_\alpha)-\alpha(y_\alpha-x_\alpha)\right)- \sF\left(\alpha(x_\alpha - x'_\alpha)\right) - \partial_t\chi(t_\alpha,x_\alpha).
\end{align*}
Since $\sF$ is assumed to be $\C^*$-increasing, by \eqref{e.y-x_in_C^*}, we obtain that
\begin{align*}
    \delta \leq \sF\left(\alpha(x_\alpha - x'_\alpha)+\nabla\chi(t_\alpha,x_\alpha)+p\right)- \sF\left(\alpha(x_\alpha - x'_\alpha)\right) - \partial_t\chi(t_\alpha,x_\alpha)
\end{align*}
for some $p$ satisfying $|p|=\eps<1$.
In case~\ref{i.case2}, by \eqref{e.cp2.|nabla_Phi|} and \eqref{e.cp2.|x_alpha-x'_alpha|}, the arguments inside $\sF$ have norms bounded by $2L+2$. In case~\ref{i.case1}, we will simply use the Lipschitzness of $\sF$. Using the definition of $V$ in \eqref{e.V=}, \eqref{e.cp2.d_tPhi>}, and \eqref{e.eps=delta/2V}, we obtain that, in both cases,
\begin{align*}
    \delta\leq V|\nabla \chi(t_\alpha,x_\alpha)|+V\eps-\partial_t\chi(t_\alpha,x_\alpha)\leq \frac{\delta}{2},
\end{align*}
reaching a contradiction. Therefore, the desired result must hold.
\end{proof}

\begin{corollary}\label{c.cp_interior}
Assume one of the following:
\begin{enumerate}[label=\rm{(\roman*)}]
    \item \label{i.case1_CF} $\sF\in\Gamma^\nearrow_\mathrm{Lip}(\cH)$, and $u,v\in\mathfrak{L}$;

    \item \label{i.case2_CF} $\sF\in\Gamma^\nearrow_\mathrm{locLip}(\cH)$, and $u,v\in\mathfrak{L}_\mathrm{Lip}$;
\end{enumerate}
If $u,v$ are a subsolution and a supersolution of $\HJ(\C,\sF)$, respectively, then $\sup_{\R_+\times\C}(u-v)=\sup_{\{0\}\times \C}(u-v)$.
\end{corollary}
\begin{proof}
Under the assumption on $u,v$, by Definition~\ref{d.vs} of viscosity solutions, we have that $u,v$ are respectively a subsolution and a supersolution of $\HJ(\mathring\C,\sF)$. This corollary immediately follows from Proposition~\ref{p.cp_H}.
\end{proof}

\begin{corollary}\label{c.cp_C}
Let $\H:\cH\to\R$ be a function.
Assume one of the following:
\begin{enumerate}[label=\rm{(\roman*)}]
    \item \label{i.case1_C} $\H\lfloor_\C\in\Gamma^\nearrow_\mathrm{Lip}(\C)$, and $u,v\in\mathfrak{M}\cap\mathfrak{L}$;

    \item \label{i.case2_C} $\H\lfloor_\C\in\Gamma^\nearrow_\mathrm{locLip}(\C)$, and $u,v\in\mathfrak{M}\cap\mathfrak{L}_\mathrm{Lip}$;
\end{enumerate}
If $u,v$ are a subsolution and a supersolution of $\HJ(\mathring\C,\H)$, respectively, then $\sup_{\R_+\times\C}(u-v)=\sup_{\{0\}\times \C}(u-v)$.
\end{corollary}
\begin{proof}
For the second case (due to $u,v\in\mathfrak{L}_\mathrm{Lip}$), we fix some $R>\sup_{t\in\R_+}\|u(t,\cdot)\|_\mathrm{Lip}\vee \|v(t,\cdot)\|_\mathrm{Lip}$. We treat both cases simultaneously. 
Applying Lemma~\ref{l.ext_loclip}, we can find $\sF\in\Gamma^\nearrow_\mathrm{Lip}(\cH)$ equal to $\H$ on $\C$ in the first case, and on $\C\cap B(0,R)$ in the second case. By Lemma~\ref{l.equiv}, $u,v$ are respectively a subsolution and a supersolution of $\HJ(\mathring\C,\sF)$ in both cases. Then, the desired result follows from Proposition~\ref{p.cp_H}.
\end{proof}
\subsection{Existence of solutions}

\begin{proposition}[Existence of solutions]\label{p.exist_sol_H}
For any $\psi\in\Gamma_\mathrm{Lip}(\C)$ and $\sF\in\Gamma_\mathrm{Lip}(\cH)$, there is a viscosity solution $f:\R_+\times\C\to\R$ of $\HJ(\mathring\C,\sF;\psi)$ in $\mathfrak{L}$.
\end{proposition}

As commented at the beginning of this section, up to this stage, we can only prove the well-posedness of $\HJ(\mathring\C,\sF)$ and $\HJ(\C,\sF)$ for $\sF\in\Gamma^\nearrow_\mathrm{Lip}(\cH)$ (see the remark below), not the desired one of $\HJ(\mathring\C,\H)$ for $\H\lfloor_\C\in\Gamma^\nearrow_\mathrm{locLip}(\C)$ in Theorem~\ref{t}. We need more properties of the solution to be proved later.

\begin{remark}\label{r.well-posed}
Let $\sF\in \Gamma^\nearrow_\mathrm{Lip}(\cH)$.
\begin{itemize}
    
    \item {\rm (Well-posedness of $\HJ(\mathring{\C},\sF)$)}
    Proposition~\ref{p.cp_H} yields the comparison principle for $\HJ(\mathring{\C},\sF)$, and Proposition~\ref{p.exist_sol_H} gives the existence of solutions of $\HJ(\mathring{\C},\sF)$ for Lipschitz $\psi$.
    
    \item {\rm (Well-posedness of $\HJ(\C,\sF)$)} Corollary~\ref{c.cp_interior} gives the comparison principle. For Lipschitz $\psi$, Proposition~\ref{p.exist_sol_H} yields a solution of $\HJ(\mathring\C,\sF;\psi)$, which, by Proposition~\ref{p.drop_bdy}, also solves $\HJ(\C,\sF;\psi)$.
\end{itemize}

\end{remark}

\begin{proof}[Proof of Proposition~\ref{p.exist_sol_H}]
We plan to produce a subsolution and a supersolution and apply Perron's method.

Fixing a positive constant $K$ satisfying $K> \sup_{B(0,\|\psi\|_\mathrm{Lip})}|\sF|$, we define $\underline u, \bar u:\R_+\times \C\to\R$ by
\begin{align*}
    \underline u (t,x) &= -Kt +\psi(x),
    \\
    \bar u (t,x) &= Kt +\psi(x),
\end{align*}
for $(t,x)\in \R_+\times \C$. Using the choice of $K$, we can verify that $\underline u$ is a subsolution and $\bar u$ is a supersolution of $\HJ(\mathring\C,\sF)$ (it is important here that the equation is over $\mathring\C$ instead of $\C$), both with initial condition $\psi$. Setting
\begin{align*}
    f(t,x)=\sup\left\{u(t,x)\ \big| \ \underline u \leq u \leq \bar u,\ \text{and $u$ is a subsolution of $\HJ(\mathring\C,\sF)$}\right\},\quad\forall (t,x)\in\R_+\times \C,
\end{align*}
we can easily verify that $f(0,\cdot) = \psi$ and $f\in\mathfrak{L}$.

We have the comparison principle for solutions of $\HJ(\mathring\C,\sF)$ in $\mathfrak{L}$ supplied by Proposition~\ref{p.cp_H} with assumption~\ref{i.case1} therein.
To apply Perron's method with this comparison principle, we need to make sure that the following holds for any function $g\in\mathfrak{L}$:
\begin{itemize}
    \item both the upper semi-continuous envelope and the lower semi-continuous envelope of $g$ belong to $\mathfrak{L}$;
    \item a ``local bump'' modification $\hat g$ of $g$ still lies in $\mathfrak{L}$.
\end{itemize}

For the first requirement, we recall that the upper semi-continuous envelop of $g$ is defined by
\begin{align*}
    \bar g(t,x) = \lim_{r\searrow 0 }\sup_{\substack{(s,y)\in\R_+\times\C \\ |(s,y)-(t,x)|\leq r }}g(s,y),\quad\forall(t,x)\in\R_+\times\C,
\end{align*}
and the lower semi-continuous envelop $\underline g$ is defined with $\sup$ in the above replaced by $\inf$.
Since $g\in\mathfrak{L}$, it is clear that $g$ is continuous at points on $\{0\}\times \C$, which implies that $\bar g(0,\cdot)  = \underline g(0,\cdot) = g(0,\cdot)$ is Lipschitz. Setting $L = \sup_{t>0}\frac{|g(t,x) -g(0,x)|}{t}$, we have, for every $(t,x) \in (0,\infty)\times\C$,
\begin{align*}
    -Lt + g(0,x) \leq g(t,x)\leq \bar g(t,x) \leq \lim_{r\searrow0}\sup_{\substack{(s,y)\in\R_+\times\C \\ |(s,y)-(t,x)|\leq r}} \left(g(0,y)+Ls\right) = g(0,x)+Lt
\end{align*}
where we used the continuity of $g(0,\cdot)$ in the last inequality.
This verifies that $\bar g\in\mathfrak{L}$. Similarly, we have $\underline g\in\mathfrak{L}$. 

For the second requirement, the local pump modification needed in Perron's method is of the following form:
\begin{align*}
    \hat g(t,x) = 
    \begin{cases}
    g(t,x) \vee v(t,x), & \quad (t,x) \in [t_0 -r, t_0+r]\times B(x_0, r),
    \\
    g(t,x), & \quad \text{otherwise},
    \end{cases}
\end{align*}
where $(t_0,x_0)\in (0,\infty) \times \mathring \C$; $r>0$ is sufficiently small so that $[t_0 -r, t_0+r]\times B(x_0, r) \subset (0,\infty) \times \mathring \C$; and $v$ is smooth on $[t_0 -r, t_0+r]\times B(x_0, r)$. 
Note that $\hat g(0,\cdot) = g(0,\cdot)$. Using $t_0-r>0$, the boundedness of $v$ over $[t_0-r,t_0+r]\times B(x_0,r)$, and $g\in\mathfrak{L}$, we can verify $\hat g\in \mathfrak{L}$.

Then, we conclude by Perron's method (\cite[Theorem~4.1]{crandall1992user} with a slight but obvious variation to adapt to the setting of a Cauchy problem) that $f$ is a viscosity solution of $\HJ(\mathring{\C},\sF;\psi)$. We have also shown that $f \in \mathfrak{L}$.
\end{proof}

\section{Monotonicity}
\label{s.monotone}

We show the monotonicity of the solution in both the spatial variable (Proposition~\ref{p.monotone}) and the temporal variable (Proposition~\ref{p.monotone_time}).

\begin{proposition}[Spatial monotonicity]\label{p.monotone}
If $f\in\mathfrak{L}_\mathrm{Lip}$ solves $\HJ(\mathring\C,\sF;\psi)$ for some  $\psi\in\Gamma^\nearrow_\mathrm{Lip}(\C)$ and some $\sF\in\Gamma^\nearrow_\mathrm{locLip}(\cH)$, then $f\in\mathfrak{M}$.
\end{proposition}

\begin{proof}
Let $\bar\psi: \cH\to\R$ be the extension of $\psi$ given in Lemma~\ref{l.ext_loclip}~\eqref{i.ext_lip}. Since $\bar \psi$ is Lipschitz, the classical result gives that there exists a Lipschitz solution $\bar f$ of $\HJ(\cH,\sF;\bar\psi)$. Straightforwardly $\bar f\lfloor_{\R_+\times \C}$ solves $\HJ(\mathring\C,\sF)$, so Proposition~\ref{p.cp_H}~\ref{i.case2} gives $\bar f\lfloor_{\R_+\times \C} =f$. Fix any $t_0\in\R_+$ and $x_0,y_0\in\C$ satisfying $x_0-y_0\in\C^*$. Set $z=x_0-y_0$ and $\bar g(t,x) = \bar f(t,x+z)$ for all $(t,x)\in\R_+\times \cH$. Then, $\bar g$ solves $\HJ(\cH,\sF)$. We set $g = \bar g\lfloor_{\R_+\times\C}$ and it is clear that $g \in \mathfrak{L}_\mathrm{Lip}$ and $g$ solves $\HJ(\mathring\C,\sF)$. Since $\bar \psi\in \Gamma^\nearrow(\cH)$ due to Lemma~\ref{l.ext_loclip}~\eqref{i.ext_lip}, we have $g(0,x) \geq f(0,x)$ for all $x\in\C$. Hence, Proposition~\ref{p.cp_H}~\ref{i.case2} yields $g\geq f$. Evaluation at $(t_0,y_0)$ gives $f(t_0,x_0)\geq f(t_0,y_0)$, which implies $f\in \mathfrak{M}$.
\end{proof}

\begin{proposition}[Temporal monotonicity]\label{p.monotone_time}
If $f \in \mathfrak{L}$ solves $\HJ(\C,\sF;\psi)$ for some $\psi\in\Gamma_\mathrm{Lip}(\C)$ and some function $\sF:\cH\to\R$ satisfying $\sF\geq 0$, then $f(t,x)\leq f(t',x)$ for all $t'\geq t\geq 0$ and $x\in\C$.
\end{proposition}

\begin{proof}
We argue by contradiction and suppose that
\begin{align*}
    \sup_{\substack{0\leq t\leq t' \\ x\in \C}}f(t,x) - f(t',x)>0.
\end{align*}
Due to $f\in\mathfrak{L}$, $\psi\in \Gamma_\mathrm{Lip}(\C)$, and $f(0,\cdot) =\psi$, the constant given by
\begin{gather*}
    K = \|\psi\|_\mathrm{Lip}\vee \left(\sup_{t>0,\,x\in\C}\frac{|f(t,x)-\psi(0,x)|}{t}\right)
\end{gather*}
is finite.
Let $\theta:\R\to\R_+$ satisfy \eqref{e.theta}.
For $T>0, R>2$ to be chosen, we set
\begin{gather*}
    \chi(x) = \theta\left(\left(1+|x|^2\right)^\frac{1}{2}-R\right),\quad\forall x\in  \C,
    \\
    \zeta(t,t',x) = \chi(x) + \frac{1}{T-t} + \frac{1}{T-t'},\quad\forall (t,t',x)\in [0,T)^2\times \C.
\end{gather*}
Choosing $T,R$ sufficiently large, we can ensure that
\begin{align}\label{e.mt.f-f-zeta>0}
    \sup_{\substack{t,t'\in A \\ x\in \C}}f(t,x) - f(t',x) - \zeta(t,t',x)>0,
\end{align}
where $A = \{(t,t'):0\leq t\leq t'\leq T\}$.
Note that, due to $R>2$, we have $\chi(0) =0$.

For each $\alpha>1$, we set
\begin{align*}
    \Psi_\alpha(t,t',x,x')= f(t,x)-f(t',x')-\alpha|x-x'|^2-\zeta(t,t',x),
    \quad\forall (t,t',x,x')\in A\times\C^2. 
\end{align*}
Note that $\Psi_\alpha(0,0,0,0)=-\frac{2}{T}$.

We show the existence of a maximizer of $\Psi_\alpha$ for each $\alpha>1$. Let $((t_{\alpha,n}, t'_{\alpha,n}, x_{\alpha,n}, x'_{\alpha,n}))_{n\in\N}$ be a maximizing sequence. Using $\Psi_\alpha(t_{\alpha,n}, t'_{\alpha,n}, x_{\alpha,n}, x'_{\alpha,n})\geq\Psi_\alpha(0,0,0,0)-1$, the definition of $K$, and the fact that $\chi(x)\geq |x|-R$, we get that
\begin{align*}
2KT + K|x_{\alpha,n} - x'_{\alpha,n}|-|x_{\alpha,n} - x'_{\alpha,n}|^2- (|x_{\alpha,n}|-R) - \frac{1}{T-t_{\alpha,n}}- \frac{1}{T-t'_{\alpha,n}} \geq -\frac{2}{T}-1.
\end{align*}
Hence, there are constants $T'<T$ and $C_1>0$ such that $(t_{\alpha,n}, t'_{\alpha,n}, x_{\alpha,n}, x'_{\alpha,n}) \in [0,T']^2\times (B(0,C_1))^2$ for all $\alpha, n$. Passing to the limit, we get a maximizer $(t_\alpha,t'_\alpha,x_\alpha,x'_\alpha)$ of $\Psi_\alpha$ which satisfies
\begin{align}\label{e.mt.bdd_x}
t_\alpha,\,t'_\alpha\in[0,T'] \quad\text{and}\quad|x_\alpha|,\, |x'_\alpha|\leq C_1,\quad\forall \alpha>1.
\end{align}

Using $\Psi_\alpha(t_\alpha, t'_\alpha, x_\alpha, x'_\alpha)\geq\Psi_\alpha(0,0,0,0)$ and the definition of $K$, we have
\begin{align*}
    2KT + K|x_\alpha - x'_\alpha|-\alpha|x_\alpha - x'_\alpha|^2 \geq -\frac{2}{T},
\end{align*}
which along with \eqref{e.mt.bdd_x} implies that
\begin{align}\label{e.mt.bdd|x-x'|}
|x_\alpha- x'_\alpha| \leq C_2 \alpha^{-\frac{1}{2}},\quad\forall \alpha>1,
\end{align}
for some constant $C_2>0$.

Due to \eqref{e.mt.bdd_x} and \eqref{e.mt.bdd|x-x'|}, up to subsequence, we have that $(t_\alpha, t'_\alpha, x_\alpha, x'_\alpha)$ converges to some $(t_\infty, t'_\infty, x_\infty, x_\infty) \in A\times \C^2$ as $\alpha\to\infty$. Since
\begin{align*}
\Psi_\alpha(t_\alpha, t'_\alpha, x_\alpha, x'_\alpha) \leq f(t_\alpha, x_\alpha) - f(t'_\alpha,x'_\alpha)-\zeta(t_\alpha,t'_\alpha,x_\alpha)
\end{align*}
and
\begin{align*}
\Psi_\alpha(t_\alpha, t'_\alpha, x_\alpha, x'_\alpha) \geq \Psi_\alpha(t,t',x,x)= f(t, x) - f(t',x)-\zeta(t,t',x),\quad\forall (t,t',x) \in A.
\end{align*}
Sending $\alpha$ to infinity and using \eqref{e.mt.f-f-zeta>0}, we obtain that
\begin{align*}
f(t_\infty, x_\infty) - f(t'_\infty,x_\infty)-\zeta(t_\infty,t'_\infty,x_\infty) >0.
\end{align*}
Since $\zeta \geq 0$ and $(t_\infty,t'_\infty)\in A$, we must have $t'_\infty > t_\infty$. Henceforth, we fix some large $\alpha$ such that $t'_\alpha > t_\alpha$. In particular $(t_\alpha,t'_\alpha)$ is in the interior of $A$ and $t'_\alpha>0$. 

Defining $\phi$ by $\Psi_\alpha(t_\alpha, t',x_\alpha,x') = \phi(t',x') - f(t',x')$, we conclude that $f-\phi$ achieves a local minimum at $(t'_\alpha,x'_\alpha)\in (0,\infty)\times \C$. Since $f$ solves $\HJ(\C,\sF)$, this implies that
\begin{align*}
    \partial_t\phi(t'_\alpha,x'_\alpha) - \sF \left(\nabla \phi(t'_\alpha,x'_\alpha) \right) \geq 0.
\end{align*}
We can compute that $\partial_t\phi(t'_\alpha,x'_\alpha) = - (T-t'_\alpha)^{-2}\leq -  T^{-2}$. Due to $\sF\geq 0$, the above display implies that $-T^{-2}\geq 0$, reaching a contradiction.
\end{proof}

\section{Lipschitzness of solutions}
\label{s.Lip}
We show that the solution of $\HJ(\C,\sF)$ is Lipschitz and obtain bounds on its Lipschitzness coefficient  (Proposition~\ref{p.lipschitz} for the spatial Lipschitzness and Proposition~\ref{p.lip_t} for the temporal Lipschitzness). In this section, we work with a Banach norm on $\cH$ comparable to the original Hilbert norm. This allows us to get results on not only the Lipschitzness measured in the Hilbert norm (Corollary~\ref{c.lipschitz}), but also the Lipschitzness measured in $l^p$ norms (Corollary~\ref{c.lip_l^p}). Recall that we are mainly interested in $\HJ(\mathring{\C},\H)$ in the statement of Theorem~\ref{t} where $\H$ possesses good properties only on $\C$. However, in this section, for convenience, we work with nonlinearities $\sF$ possessing good properties on $\cH$. Using results from Section~\ref{s.prelim}, we can later deduce the Lipschitzness of solutions of $\HJ(\mathring{\C},\H)$ as done in the proof of Theorem~\ref{t}~\eqref{i.main_lip} in Section~\ref{s.intro_main_result}.

A norm $\|\cdot\|$ on $\cH$ is said to be \textit{comparable} (to $|\cdot|$) if there is $C>0$ such that
\begin{align*}
    C^{-1}\|x\| \leq |x|\leq C\|x\|,\quad\forall x\in\cH.
\end{align*}
If $\|\cdot\|$ is differentiable at $x\in\cH$, we denote its differential at $x$ by $\nabla \|x\|$. For the Lipschitzness in the spatial variable with respect to $\|\cdot\|$, we need to impose the following condition:
\begin{align}\label{e.||||_diff}
     \text{$\|\cdot\|$ is differentiable at every $x\in\cH\setminus\{0\}$, and } \sup_{x\in \cH\setminus\{0\}}|\nabla\|x\||<\infty.
\end{align}
For any $g:\C\to\R$, define
\begin{align*}
    \|g\|_{\mathrm{Lip}\|\cdot\|} = \sup_{\substack{y,y'\in\C\\ y\neq y'}}\frac{|g(y)-g(y')|}{\|y-y'\|}.
\end{align*}
For consistency with other parts of the paper, we write $\|g\|_\mathrm{Lip} = \|g\|_{\mathrm{Lip}\,|\cdot|}$.

We denote by $\|\cdot\|_*$ the norm dual to $\|\cdot\|$ with respect to the inner product, which is given by
\begin{align}\label{e.dual}
    \|x\|_* = \sup_{y\in\cH,\,\|y\|\leq 1}\la y,x\ra,\quad\forall x\in\cH.
\end{align}

\subsection{Lipschitzness in the spatial variable}

\begin{proposition}[Spatial Lipschitzness]\label{p.lipschitz}
Let $\|\cdot\|$ be comparable and satisfy \eqref{e.||||_diff}.
If $f\in \mathfrak{L}$ solves $\HJ(\C,\sF;\psi)$ for some $\psi\in\Gamma_\mathrm{Lip}(\C)$ and $\sF\in\Gamma_\mathrm{locLip}(\cH)$, then $\sup_{t\in\R_+}\|f(t,\cdot)\|_{\mathrm{Lip}\|\cdot\|} =\|\psi \|_{\mathrm{Lip}\|\cdot\|}$.
\end{proposition}

We comment that, if the norm is $|\cdot|$, \rv{under the additional assumption $\sF\in\Gamma^\nearrow_\mathrm{locLip}(\cH)$,} there is a simpler proof than the one below. We can use Lemma~\ref{l.ext_loclip}~\eqref{i.ext_lip} to find an extension $\bar \psi$ of $\psi$ to $\cH$ with $\|\bar \psi\|_\mathrm{Lip}=\|\psi\|_\mathrm{Lip}$. Let $\bar f$ solve $\HJ(\cH,\sF;\bar\psi)$. It is classical that $\sup_{t\in\R_+}\|\bar f(t,\cdot)\|_\mathrm{Lip}= \|\bar \psi\|_\mathrm{Lip}$. Since both $\bar f$ and $f$ solves $\HJ(\mathring\C,\sF;\psi)$, the comparison principle (Proposition~\ref{p.cp_H}~\ref{i.case2}) implies $\bar f\lfloor_{\R_+\times\C}=f$. Then, for any $z\in\cH$, we consider translations $\bar g^\pm = \bar f(\cdot,\cdot+z)\pm \|\psi\|_\mathrm{Lip}|z|$ and the restrictions $g^\pm$ of $\bar g^\pm$ to $\R_+\times \C$ which solve $\HJ(\mathring\C,\sF)$. Proposition~\ref{p.cp_H}~\ref{i.case2} yields $g^-\leq f\leq g^+$, from which we can deduce the announced result.

If the norm is more general, we do not know the existence of an extension of $\psi$ that preserves the Lipschitz coefficient. Hence, we need a more convoluted argument.

\begin{proof}
Since $\psi$ is Lipschitz (with respect to $|\cdot|$) and $\|\cdot\|$ is comparable, we have $\|\psi\|_{\mathrm{Lip}\|\cdot\|}<\infty$.
Set $L = \|\psi\|_{\mathrm{Lip}\|\cdot\|}$. We argue by contradiction and assume that there is $(\underline t, \underline x, \underline x')\in (0,\infty)\times \C\times\C$ such that
\begin{align*}
    f\left(\underline t,\underline x\right) - f\left(\underline t,\underline x'\right) > L\|\underline x-\underline x'\|.
\end{align*}
Since $f(0,\cdot) = \psi$, we have
\begin{align}\label{e.pre_lip.contr_lip_1}
    \sup_{\substack{t\in\R_+ \\ x,x'\in\C}}f\left(t,x\right) - f\left(t,x'\right) - L\|x-x'\|> 0\geq \sup_{x,x'\in \C} f\left(0,x\right) - f\left(0,x'\right) - L\|x-x'\|.
\end{align}
Due to $f\in \mathfrak{L}$ and $f(0,\cdot)=\psi$, there is a constant $K>0$ such that
\begin{align}\label{e.pre_lip.|f-psi|<Kt}
    \sup_{x\in\C}|f(t,x) - \psi(x)|\leq Kt, \quad\forall t\in\R_+.
\end{align}

Choosing $T>0$ sufficiently large, we can modify \eqref{e.pre_lip.contr_lip_1} into
\begin{align}
    \sup_{\substack{t\in [0,T) \\ x,x'\in\C}}f\left(t,x\right) - f\left(t,x'\right)  - L\|x-x'\|-\frac{2}{T-t} \notag
    \\
    > 0 > \sup_{x,x'\in \C} f\left(0,x\right) - f\left(0,x'\right) - L\|x-x'\|-\frac{2}{T}. \label{e.pre_lip.contr_lip_2}
\end{align}
We set
\begin{gather}
    C = \sup_{x\in \cH\setminus\{0\}}|\nabla\|x\||<\infty, \notag
    \\
    V=\sup_{\substack{y,y'\in B(0,LC+1) \\ y\neq y'}}\frac{|\sF(y)-\sF(y')|}{|y-y'|}. \label{e.pre_lip.V>|F|}
\end{gather}
Let $\theta:\R\to\R_+$ satisfy \eqref{e.theta}.
For $\delta,R>0$ to be chosen, we set, for $(t,t',x,x')\in \R_+^2\times\C^2$,
\begin{align*}
    \zeta_1(t,t') & = \frac{1}{T-t} + \frac{1}{T-t'},
    \\
    \chi(t,x) & = \theta\left((1+|x|^2)^\frac{1}{2}+Vt-R\right)
    \\
    \zeta_2(t,t',x,x') & = \chi(t,x) +\chi(t',x').
\end{align*}
Building on \eqref{e.pre_lip.contr_lip_2}, we choose $\delta\in(0,1)$ sufficiently small and $R>2$ sufficiently large so that
\begin{align}
    \sup_{\substack{t\in [0,T) \\ x,x'\in\C}}f\left(t,x\right) - f\left(t,x'\right)  - L\|x-x'\|-\delta t-\zeta_1(t,t)-\zeta_2(t,t,x,x') \notag
    \\
    > 0 > \sup_{x,x'\in \C} f\left(0,x\right) - f\left(0,x'\right) - L\|x-x'\|-\zeta_1(0,0)-\zeta_2(0,0,x,x'). \label{e.pre_lip.contr_lip_3}
\end{align}
Before proceeding, we record the following useful results: for all $t,t',x,x'$,
\begin{align}
    \zeta_2(t,t',x,x') & \geq |x|+|x'|-2R, \label{e.pre_lip.zeta_2>|x|+|x'|-2R}
    \\
    |\nabla \chi(t,x)| & \leq 1, \label{e.pre_lip.nabla_chi<1}
    \\
    \partial_t\chi(t,x) & \geq V |\nabla \chi(t,x)|. \label{e.pre_lip.dt_chi>V|grad_chi|}
\end{align}
For each $\alpha>1$, we define, for $(t,t',x,x') \in [0,T)^2\times \C^2$,
\begin{align*}
    \Psi_\alpha(t,t',x,x') = f(t,x) - f(t',x') - L\|x-x'\|-\delta t- \alpha|t-t'|^2
    -\zeta_1(t,t')-\zeta_2(t,t',x,x').
\end{align*}

We show the existence of a maximizer of $\Psi_\alpha$ for each $\alpha>1$.
Let $((t_{\alpha,n},t'_{\alpha,n},x_{\alpha,n},x'_{\alpha,n}))_{n=1}^\infty$ be a maximizing sequence. Note that, due to $R>2$, $\zeta_2(0,0,0,0)=0$. Hence,
it is clear from the definition that
\begin{align*}
    \Psi_\alpha(0,0,0,0) = -2T^{-1},\quad\forall \alpha>1.
\end{align*}
Also, by \eqref{e.pre_lip.|f-psi|<Kt} and the definition of $L$, we have
\begin{align}\label{e.pre_lip.f-f-L}
    f(t,x) - f(t',x') -L\|x-x'\|  \leq 2KT,\quad\forall (t,t',x,x')\in[0,T)^2\times \C^2.
\end{align}
Using these together with \eqref{e.pre_lip.zeta_2>|x|+|x'|-2R} and $\Psi_\alpha (t_{\alpha,n},t'_{\alpha,n},x_{\alpha,n},x'_{\alpha,n})\geq \Psi_\alpha(0,0,0,0)-1$ for sufficiently large $n$, we get
\begin{align*}
    2KT -\zeta_1(t_{\alpha,n},t'_{\alpha,n})- (|x_{\alpha,n}|+|x'_{\alpha,n}|-2R)  \geq - 2T^{-1}-1.
\end{align*}
Hence, there are constants $T'\in(0,T)$ and $C_1>0$ such that
\begin{align*}
    t_{\alpha,n},\,t'_{\alpha,n}\in [0,T'],\qquad |x_{\alpha,n}|,\,|x'_{\alpha,n}|\leq C_1
\end{align*}
for sufficiently large $n$ and every $\alpha$. Hence, we can conclude the existence of a maximizer of $\Psi_\alpha$ denoted by $(t_\alpha, t'_\alpha, x_\alpha,x'_\alpha)$, which also satisfies
\begin{align*}
    t_\alpha,\,t'_\alpha\in [0,T'],\qquad |x_\alpha|,\,|x'_\alpha|\leq C_1.
\end{align*}

Using $\Psi_\alpha(t_\alpha, t'_\alpha, x_\alpha,x'_\alpha)\geq \Psi_\alpha(0,0,0,0)= -2T^{-1}$ and \eqref{e.pre_lip.f-f-L}, we get
\begin{align*}
    2KT - \alpha|t_\alpha-t'_\alpha|^2 \geq -2T^{-1},
\end{align*}
which implies that $\lim_{\alpha\to\infty}|t_\alpha - t'_\alpha| =0$. Hence, up to a subsequence, we have that $(t_\alpha,t'_\alpha,x_\alpha,x'_\alpha)$ converges to some $(t_\infty,t_\infty,x_\infty,x'_\infty)$ as $\alpha\to\infty$. Since
\begin{align*}
    \Psi_\alpha(t_\alpha,t'_\alpha,x_\alpha,x'_\alpha)\leq f(t_\alpha,x_\alpha) - f(t'_\alpha,x'_\alpha)-L\|x_\alpha-x'_\alpha\|-\delta t_\alpha
    - \zeta_1(t_\alpha,t'_\alpha)-\zeta_2(t_\alpha,t'_\alpha,x_\alpha,x'_\alpha),
\end{align*}
and
\begin{align*}
    \Psi_\alpha(t_\alpha,t'_\alpha,x_\alpha,x'_\alpha) \geq \Psi_\alpha(t,t,x,x')
    = f(t,x)-f(t,x')-L\|x-x'\|-\delta t
    \\
    -\zeta_1(t,t)-\zeta_2(t,t,x,x')
\end{align*}
for every $(t,x,x')\in [0,T)\times\C\times\C$, by sending $\alpha\to\infty$, we get that
\begin{align*}
    f(t_\infty,x_\infty) - f(t_\infty,x'_\infty)-L\|x_\infty-x'_\infty\|-\delta t_\infty
    - \zeta_1(t_\infty,t_\infty)-\zeta_2(t_\infty,t_\infty,x_\infty,x'_\infty) 
\end{align*}
is greater than or equal to the left-hand side in \eqref{e.pre_lip.contr_lip_3}. Therefore, by \eqref{e.pre_lip.contr_lip_3}, we must have $t_\infty>0$. Also, the above is strictly negative if $x_\infty = x'_\infty$, while the left-hand side in \eqref{e.pre_lip.contr_lip_3} is strictly positive. Hence, we must have $x_\infty\neq x_\infty$. In conclusion, for sufficiently large $\alpha$, we have $t_\alpha,\,t'_\alpha>0$ and $x_\alpha \neq x'_\alpha$.

Fix any such $\alpha$. Using that $(t,x)\mapsto \Psi_\alpha(t,t'_\alpha,x,x'_\alpha)$ and $(t',x')\mapsto \Psi_\alpha(t_\alpha,t',x_\alpha,x')$ achieve local maximums at $(t_\alpha,x_\alpha)$ and $(t'_\alpha,x'_\alpha)$, respectively, and using that $f$ solves $\HJ(\C,\sF)$, we have
\begin{align*}
    \delta + 2\alpha(t_\alpha-t'_\alpha) + \frac{1}{(T-t_\alpha)^2}+\partial_t\chi(t_\alpha,x_\alpha) - \sF\left(L\nabla\|x_\alpha-x'_\alpha\|+\nabla\chi(t_\alpha,x_\alpha)\right)\leq 0,
    \\
    2\alpha(t_\alpha-t'_\alpha) - \frac{1}{(T-t'_\alpha)^2}-\partial_t\chi(t'_\alpha,x'_\alpha) - \sF\left(L\nabla\|x_\alpha-x'_\alpha\|-\nabla\chi(t'_\alpha,x'_\alpha)\right)\geq 0.
\end{align*}
By assumption on the differential of $\|\cdot\|$ and \eqref{e.pre_lip.nabla_chi<1}, the norms (in $|\cdot|$) of terms inside $\sF$ in the above are bounded by $LC+1$. Taking the difference of the above relations, and using \eqref{e.pre_lip.V>|F|} and \eqref{e.pre_lip.dt_chi>V|grad_chi|}, we obtain $\delta\leq 0$, which is absurd. Therefore, the desired result must hold.
\end{proof}

\subsection{Lipschitzness in the temporal variable}

Recall the notation for the norm dual to $\|\cdot\|$ in~\eqref{e.dual}. For the supremum on the right-hand side of \eqref{e.sup_lip_t} to be taken over a subset of $\C$, we need to restrict $f-\phi$ to achieve a local extremum only in the interior for a smooth test function $\phi$. Hence, the following proposition is stated for $\HJ(\mathring{\C},\sF)$. Since a solution of $\HJ(\C,\sF)$ is obviously a solution of $\HJ(\mathring{\C},\sF)$ (see Definition~\ref{d.vs}), the following result holds for solutions of $\HJ(\C,\sF)$.

\begin{proposition}[Temporal Lipschitzness]\label{p.lip_t}
Let $\|\cdot\|$ be comparable.
If $f\in \mathfrak{M}\cap \mathfrak{L}$ solves $\HJ(\mathring\C,\sF;\psi)$ for some $\psi\in \Gamma_\mathrm{Lip}(\C)$ and some locally bounded $\sF:\cH\to\R$, and $\sup_{t\in\R_+}\|f(t,\cdot)\|_{\mathrm{Lip}\|\cdot\|} = \|\psi\|_{\mathrm{Lip}\|\cdot\|}$, then
 \begin{align}\label{e.sup_lip_t}
     \sup_{x\in\C}\|f(\cdot,x)\|_\mathrm{Lip} \leq \sup_{\substack{v\in\C\\ \|v\|_* \leq \|\psi\|_{\mathrm{Lip}\|\cdot\|}}}|\sF(v)|
 \end{align}
\end{proposition}
\begin{proof}
We set $L$ to be the right-hand side of \eqref{e.sup_lip_t}.
We argue by contradiction and assume that there is $(\underline t, \underline t', \underline x)\in \R_+\times\R_+\times\C$ such that
\begin{align}\label{e.lt.contr_lip_1}
    f\left(\underline t,\underline x\right) - f\left(\underline t',\underline x\right) - L|\underline t-\underline t'| > 0.
\end{align}

By the assumption on $f$, we have that
\begin{align*}
    K = \left(\sup_{t\in\R_+}\|f(t,\cdot)\|_{\mathrm{Lip}}\right)\vee \left(\sup_{t>0,\ x\in\C}\frac{|f(t,x)-f(0,x)|}{t}\right)
\end{align*}
is finite, where we used the comparability of $\|\cdot\|$ to ensure the finiteness of the first term on the right.

Choosing $T>0$ sufficiently large and $\delta>0$ sufficiently small, we can obtain from \eqref{e.lt.contr_lip_1} that
\begin{align*}
    \sup_{\substack{t,t'\in [0,T) \\ x\in\C}}f\left(t,x\right) - f\left(t',x\right)  - L|t-t'|-\zeta_1(t,t') -\frac{\delta}{d(x)}-\delta|x| >0 
\end{align*}
where $d = d_\C$ defined in \eqref{e.d} and
\begin{gather*}  
\zeta_1(t,t')  = \frac{1}{T-t} + \frac{1}{T-t'},\quad\forall (t,t')\in [0,T)^2,
\end{gather*}

Let $\theta:\R\to\R_+$ satisfy \eqref{e.theta}.
For $R>0$ to be chosen, we set
\begin{align*}
    \zeta_2(x) & = \theta\left((1+|x|^2)^\frac{1}{2}-R\right),\quad x\in\C.
\end{align*}
Fix any $y_0\in \mathring \C$, and set
\begin{align*}
    C_0 = \frac{2}{T}+\frac{\delta}{d(y_0)} +  |y_0|.
\end{align*}
We choose $R>0$ sufficiently large so that
\begin{align}\label{e.lt.zeta_2(y_0)=0}
    \zeta_2(y_0) =0
\end{align}
and
\begin{align}
    \sup_{\substack{t,t'\in [0,T) \\ x\in\C}}f\left(t,x\right) - f\left(t',x\right)  - L|t-t'|-\zeta_1(t,t')-\zeta_2(x)-\frac{\delta}{d(x)}-\delta |x| >0. \label{e.lt.contr_lip_3}
\end{align}
Also, note that
\begin{align}
    \zeta_2(x) & \geq |x|-R,\quad\forall x\in\C. \label{e.lt.zeta_2>|x|+|x'|-2R}
\end{align}

For each $\alpha>1$, we define, for $(x,y)\in\C^2$,
\begin{align*}
    \zeta_3(x,y) = \alpha|x-y|^2 +\frac{\delta}{d(y)}+\delta |y|,
\end{align*}
and, for $(t,t',x,x',y) \in [0,T)^2\times \C^3$,
\begin{align*}
    \Psi_\alpha(t,t',x,x',y) = f(t,x) - f(t',x') - L|t-t'|- \alpha|x-x'|^2
    -\zeta_1(t,t')-\zeta_2(x)-\zeta_3(x,y).
\end{align*}

We show the existence of a maximizer of $\Psi_\alpha$ for each $\alpha>1$.
Let $((t_{\alpha,n},t'_{\alpha,n},x_{\alpha,n},x'_{\alpha,n},y_{\alpha,n}))_{n=1}^\infty$ be a maximizing sequence.
Using \eqref{e.lt.zeta_2(y_0)=0}, we have
\begin{align*}
    \Psi_\alpha(0,0,y_0,y_0,y_0) =-C_0,\quad\forall \alpha>1.
\end{align*}
Also, by the definitions of $K$, we have
\begin{align}\label{e.lt.f-f-L}
    f(t,x) - f(t',x')  \leq 2KT+K|x-x'|,\quad\forall (t,t',x,x')\in[0,T)^2\times \C^2.
\end{align}
Using these together with $\alpha>1$, \eqref{e.lt.zeta_2>|x|+|x'|-2R}, and $\Psi_\alpha (t_{\alpha,n},t'_{\alpha,n},x_{\alpha,n},x'_{\alpha,n},y_{\alpha,n})\geq \Psi_\alpha(0,0,y_0,y_0,y_0)-1$ for sufficiently large $n$, we get
\begin{align*}
    2KT+K|x_{\alpha,n}-x'_{\alpha,n}|-|x_{\alpha,n}-x'_{\alpha,n}|^2 -\zeta_1(t_{\alpha,n},t'_{\alpha,n})- (|x_{\alpha,n}|-R) - \delta|y_{\alpha,n}| \geq - C_0-1.
\end{align*}
Hence, there are constants $T'\in(0,T)$ and $C_1>0$ such that
\begin{align*}
    t_{\alpha,n},\,t'_{\alpha,n}\in [0,T'],\qquad |x_{\alpha,n}|,\,|x'_{\alpha,n}|,\,|y_{\alpha,n}|\leq C_1
\end{align*}
for sufficiently large $n$ and every $\alpha$. So, we can conclude the existence of a maximizer of $\Psi_\alpha$ denoted by $(t_\alpha, t'_\alpha, x_\alpha,x'_\alpha, y_\alpha)$, which also satisfies
\begin{align}\label{e.lt.tt,xx'yy'_bdd}
    t_\alpha,\,t'_\alpha\in [0,T'],\qquad |x_\alpha|,\,|x'_\alpha|,\,|y_\alpha|\leq C_1.
\end{align}

Using $\Psi_\alpha(t_\alpha, t'_\alpha, x_\alpha,x'_\alpha, y_\alpha)\geq \Psi_\alpha(t_\alpha, t'_\alpha, x_\alpha,x_\alpha, y_\alpha)$, we get
\begin{align*}
    -f(t'_\alpha,x'_\alpha) - \alpha|x_\alpha-x'_\alpha|^2 \geq -f(t'_\alpha,x_\alpha),
\end{align*}
which along with the definition of $K$ implies that
\begin{align}\label{e.lt.|x-x'|}
    |x_\alpha-x'_\alpha|\leq K\alpha^{-1},\quad\forall \alpha>1.
\end{align}
Using $\Psi_\alpha(t_\alpha, t'_\alpha, x_\alpha,x'_\alpha, y_\alpha)\geq \Psi_\alpha(0,0,y_0,y_0,y_0) =-C_0$ and \eqref{e.lt.f-f-L}, we have
\begin{align*}
    2KT + K |x_\alpha-x'_\alpha|-\frac{\delta}{d(y_\alpha)}\geq -C_0
\end{align*}
which implies that, for some constant $C_2>0$,
\begin{align}\label{e.lt.d(y),d(y')>}
    d(y_\alpha) >C_2,\quad\forall \alpha>1.
\end{align}
Therefore, $y_\alpha\in\mathring \C$ for all $\alpha>1$. In particular, $y_\alpha\neq 0$.

Since $y\mapsto \Psi_\alpha(t_\alpha, t'_\alpha,x_\alpha,x'_\alpha,y)$ has a local maximum at $y_\alpha$, we can verify using Lemma~\ref{l.d}~\eqref{i.g-1/d} that
\begin{align*}
    \left(d(y_\alpha)\right)^2\left(2\delta^{-1}\alpha(y_\alpha-x_\alpha) + \frac{y_\alpha}{|y_\alpha|}\right) \in D^+d(y_\alpha).
\end{align*}
Hence, due to \eqref{e.lt.d(y),d(y')>} and Lemma~\ref{l.d}~\eqref{i.Dd}, there is a constant $C_3>0$ such that
\begin{align}
    |x_\alpha-y_\alpha|\leq C_3\alpha^{-1},\quad\forall \alpha>1.  \label{e.lt.2alpha|x-y|}
\end{align}

Using \eqref{e.lt.tt,xx'yy'_bdd}, \eqref{e.lt.|x-x'|}, and \eqref{e.lt.2alpha|x-y|}, we have that $(t_\alpha,t'_\alpha,x_\alpha,x'_\alpha, y_\alpha)$ converges along some subsequence to some $(t_\infty,t'_\infty,x_\infty,x_\infty, x_\infty)$ as $\alpha\to\infty$. Since
\begin{align*}
    \Psi_\alpha(t_\alpha,t'_\alpha,x_\alpha,x'_\alpha, y_\alpha)\leq f(t_\alpha,x_\alpha) - f(t'_\alpha,x'_\alpha)-L|t_\alpha-t'_\alpha|
    - \zeta_1(t_\alpha,t'_\alpha)-\zeta_2(x_\alpha) -\frac{\delta}{d(y_\alpha)}-\delta|y_\alpha|,
\end{align*}
and
\begin{align*}
    \Psi_\alpha(t_\alpha,t'_\alpha,x_\alpha,x'_\alpha, y_\alpha) &\geq \Psi_\alpha(t,t',x,x,x)
    \\
    &= f(t,x)-f(t',x)-L|t-t'|
    -\zeta_1(t,t')-\zeta_2(x)-\frac{\delta}{d(x)}-\delta|x|
\end{align*}
for every $(t,t',x)\in [0,T)^2\times\C$, by sending $\alpha\to\infty$ and using \eqref{e.lt.contr_lip_3}, we get that
\begin{align*}
    f(t_\infty,x_\infty) - f(t'_\infty,x_\infty)-L|t_\infty-t'_\infty|
    - \zeta_1(t_\infty,t'_\infty)-\zeta_2(x_\infty) -\frac{\delta}{d(x_\infty)}-\delta|x_\infty| >0.
\end{align*}
Hence, we must have $t_\infty\neq t'_\infty$ because otherwise, the left-hand side of the above display will be nonpositive. In view of this, \eqref{e.lt.|x-x'|}, \eqref{e.lt.d(y),d(y')>} and \eqref{e.lt.2alpha|x-y|}, we can fix some large $\alpha>0$ so that $t_\alpha\neq t'_\alpha$ and $x_\alpha,x'_\alpha,y_\alpha \in \mathring \C$. 

If $t_\alpha > t'_\alpha$ which implies that $t_\alpha>0$, since $f$ solves $\HJ(\mathring\C,\sF)$ and since $(t,x)\mapsto\Psi_\alpha(t,t'_\alpha,x,x'_\alpha,y_\alpha)$ achieves a local maximum at $(t_\alpha,x_\alpha)$, we can get
\begin{align}\label{e.dt-H<0}
    \partial_t\phi(t_\alpha,x_\alpha) - \sF\left(\nabla\phi(t_\alpha,x_\alpha)\right)\leq 0,
\end{align}
where $\phi$ is given by $\Psi_\alpha(t,t'_\alpha,x,x'_\alpha,y_\alpha)=f(t,x)-\phi(t,x)$. 
Then, we show that
\begin{align}\label{e.|grad_phi|<}
    \nabla\phi(t_\alpha,x_\alpha) \in\C,\qquad \| \nabla\phi(t_\alpha,x_\alpha)\|_* \leq  \|\psi\|_{\mathrm{Lip}\|\cdot\|}.
\end{align}
Since $f-\phi$ achieves a local maximum at $(t_\alpha,x_\alpha)$ and $x_\alpha\in\mathring\C$, we have that
\begin{align*}
    \phi(t_\alpha,x)-\phi(t_\alpha,x_\alpha)\geq f(t_\alpha,x)-f(t_\alpha,x_\alpha),
\end{align*}
for all $x\in B(x_\alpha,r)$ for some sufficiently small $r>0$. Since $f\in \mathfrak{M}$ implies that $f(t_\alpha,\cdot)$ is $\C^*$-increasing, replacing $x$ by $x_\alpha + \eps y$ for $y\in\C^*$ and sufficiently small $\eps$, and sending $\eps\to0$, we obtain $\la y ,\nabla\phi(t_\alpha,x_\alpha)\ra\geq 0$ for all $y\in\C^*$, which implies the first part of \eqref{e.|grad_phi|<} by duality. By the assumption on $f$, the right-hand side of the above display is greater or equal to $-\|\psi\|_{\mathrm{Lip}\|\cdot\|}\|x-x_\alpha\|$. Varying $x \in B(x_\alpha, r)$ and using the comparability of $\|\cdot\|$ to see the second half of \eqref{e.|grad_phi|<} (where the comparability is needed because $B(x_\alpha,r)$ is a ball with respect to $|\cdot|$).

Now, let us conclude. Due to $t_\alpha>0$, we can compute that
\begin{align*}
    \partial_t\phi(t_\alpha,x_\alpha) = L+(T-t_\alpha)^{-2}\geq L+ T^{-2}.
\end{align*}
Form this, \eqref{e.dt-H<0},~\eqref{e.|grad_phi|<}, and the definition of $L$, we can deduce $T^{-2}\leq 0$, reaching a contradiction.

If $t'_\alpha>t_\alpha$ which implies that $t'_\alpha>0$, since $(t',x')\mapsto\Psi_\alpha(t_\alpha,t',x_\alpha,x',y_\alpha)$ achieves a local maximum at $(t'_\alpha,x'_\alpha)$ and since $f$ is a supersolution, we have
\begin{align}\label{e.-(L+delta)-H}
    \partial_t\tilde\phi(t'_\alpha,x'_\alpha)-\sF\left(\nabla\tilde\phi(t'_\alpha,x'_\alpha)\right)\geq 0,
\end{align}
for $\tilde\phi$ given by $\Psi_\alpha(t_\alpha,t',x_\alpha,x',y_\alpha)=\tilde\phi(t',x')-f(t',x')$. Now, since $f-\tilde\phi$ achieves a local minimum at $(t'_\alpha,x'_\alpha)$ and $x'_\alpha\in\mathring\C$, we can derive that $\nabla\tilde\phi(t'_\alpha,x'_\alpha)$ satisfies the same relations as in \eqref{e.|grad_phi|<}. Since $\partial_t\tilde\phi(t'_\alpha,x'_\alpha)= -L-(T-t'_\alpha)^{-2}\leq -L- T^{-2}$, this along with~\eqref{e.-(L+delta)-H} and the definition of $L$ yields $- T^{-2}\geq 0$, reaching a contradiction again.
\end{proof}

\subsection{Corollaries}

An immediate corollary is the application of Propositions~\ref{p.lipschitz} and~\ref{p.lip_t} to the original Hilbert norm $|\cdot|$:
\begin{corollary}[Lipschitzness in the Hilbert norm]\label{c.lipschitz}
If $f\in \mathfrak{M}\cap\mathfrak{L}$ solves $\HJ(\C,\sF;\psi)$ for some $\psi\in\Gamma_\mathrm{Lip}(\C)$ and $\sF\in\Gamma_\mathrm{locLip}(\cH)$, then $\sup_{t\in\R_+}\|f(t,\cdot)\|_{\mathrm{Lip}} =\|\psi \|_{\mathrm{Lip}}$ and $\sup_{x\in\C}\|f(\cdot,x)\|_\mathrm{Lip} \leq \sup_{\C\cap B(0,\|\psi\|_\mathrm{Lip})}|\sF|$.
\end{corollary}
\noindent Note that if $f$ solves $\HJ(\C,\sF;\psi)$ then obviously $f$ solves $\HJ(\mathring\C,\sF;\psi)$ and thus Proposition~\ref{p.lip_t} is applicable here.

Another corollary concerns the setting where $\cH$ is a product space and $\|\cdot\|$ is an $l^p$ norm:
\begin{enumerate}[start=1,label={\rm{(P)}}]
    \item \label{i.H_prod} Let $\cH = \times_{i=1}^k\cH_i$ where each $\cH_i$ is a Hilbert space with inner product $\la\cdot,\cdot\ra_{\cH_i}$ and the induced norm $|\cdot|_{\cH_i}$. 
    Let $a_1,a_2,\ldots, a_k>0$ satisfy $\sum_{i=1}^ka_i=1$. We set $\la x,x'\ra_\cH= \sum_{i=1}^k a_i \la x_i,x'_i\ra_{\cH_i}$ for $x,x'\in\cH$.
    We define
    \begin{align*}
        \|x\|_p = \left(\sum_{i=1}^ka_i|x_i|^p_{\cH_i}\right)^\frac{1}{p},
    \end{align*}
    for $p\in[1,\infty)$,
    and $\|x\|_\infty = \sup_{i=1,2,\ldots,k}|x_i|_{\cH_i}$ for all $x\in \cH$. As usual, we set $p^* = \frac{p}{p-1}$.
\end{enumerate}
It is clear that $\|\cdot\|_{p^*}$ is dual to $\|\cdot\|_p$.

\begin{corollary}[Lipschitzness in $l^p$ norms]\label{c.lip_l^p}
Under~\ref{i.H_prod}, if $f\in\mathfrak{M}\cap\mathfrak{L}$ solves $\HJ(\C,\sF;\psi)$ for some $\psi\in\Gamma_\mathrm{Lip}(\C)$ and $\sF\in\Gamma_\mathrm{locLip}(\cH)$, then $\sup_{t\in\R_+}\|f(t,\cdot)\|_{\mathrm{Lip}\|\cdot\|_p} = \|\psi \|_{\mathrm{Lip}\|\cdot\|_p}$ and $\sup_{x\in\C}\|f(\cdot,x)\|_\mathrm{Lip}\leq\sup_{v\in\C,\,\|v\|_{p^*}\leq \|\psi\|_{\mathrm{Lip}\|\cdot\|_p}}|F(v)|$ for all $p\in[1,\infty]$.
\end{corollary}
\begin{proof}
In view of Proposition~\ref{p.lip_t}, it suffices to show the uniform Lipschitzness in the spatial variable.

Since $\|\cdot\|_p$ is comparable with $|\cdot|$ for all $p\in[1,\infty]$, and $\|\cdot\|_p$ is differentiable on $\cH\setminus\{0\}$ with bounded differential for all $p\in(1,\infty)$, the desired result for $p\in(1,\infty)$ follows from Proposition~\ref{p.lipschitz}. For $p\in\{1,\infty\}$, we can use the continuity of $\|\cdot\|_p$ as $p\to1$ and $p\to\infty$ to conclude. Indeed, setting $\underline a = \min_i a_i$, we have,
\begin{align*}
    \underline a^\frac{1}{p}\|\cdot\|_\infty \leq \|\cdot\|_p\leq \|\cdot\|_\infty,
\end{align*}
which implies that, for any Lipschitz $g$,
\begin{align*}
    \|g \|_{\mathrm{Lip}\|\cdot\|_\infty} \leq \|g \|_{\mathrm{Lip}\|\cdot\|_p}\leq \underline a^{-\frac{1}{p}}\|g \|_{\mathrm{Lip}\|\cdot\|_\infty}.
\end{align*}
Sending $p\to\infty$, we obtain that $\lim_{p\to\infty}\|g \|_{\mathrm{Lip}\|\cdot\|_p} = \|g \|_{\mathrm{Lip}\|\cdot\|_\infty}$, from which we can deduce the desired result for $p=\infty$. 
Now, we turn to the case $p=1$.
Since $|g(x)-g(x')|\leq \|g \|_{\mathrm{Lip}\|\cdot\|_p}\|x-x'\|_p$, sending $p\to1$, we have $\liminf_{p\to1}\|g \|_{\mathrm{Lip}\|\cdot\|_p} \geq \|g \|_{\mathrm{Lip}\|\cdot\|_1}$. On the other hand, due to $\|\cdot\|_1\leq \|\cdot\|_p$, we have $\|g \|_{\mathrm{Lip}\|\cdot\|_1}\geq \|g \|_{\mathrm{Lip}\|\cdot\|_p}$. Therefore, we must have $\lim_{p\to1}\|g \|_{\mathrm{Lip}\|\cdot\|_p} = \|g \|_{\mathrm{Lip}\|\cdot\|_1}$, from which the result for $p=1$ follows.
\end{proof}

\section{Variational representations of solutions}
\label{s.var}
We show that the solution of $\HJ(\mathring\C,\H;\psi)$ can be represented by the Hopf--Lax formula~\eqref{e.hopf_lax} if $\H$ is convex (Proposition~\ref{p.hopf-lax}), or by the Hopf formula~\eqref{e.hopf} if $\psi$ is convex (Proposition~\ref{p.hopf}). Let us introduce the necessary notations and definitions.
For $\cD\supset\C$ and $g:\cD\to(-\infty,\infty]$, we define the \textit{monotone convex conjugate} (over $\C$) of $g$ by
\begin{align}\label{e.def_u*}
    g^*(y) = \sup_{x\in \C}\{\la x,y\ra_\cH-g(x)\},\quad\forall y \in \cH.
\end{align}
Here, the qualifier ``montone'' is added because, as a result of the supremum taken over $\C$, the function $g^*:\cH\to(-\infty,\infty]$ is $\C^*$-increasing. We denote the \textit{biconjugate} of $g$ by $g^{**}=(g^*)^*$.

\begin{definition}\label{d.fenchel_moreau_prop}
A nonempty closed convex cone $\C$ is said to possess the \textit{Fenchel--Moreau property} if the following holds: for every $g:\C\to(-\infty,\infty]$ not identically equal to $\infty$, we have that $g^{**}=g$ on $\C$ if and only if $g$ is convex, lower semicontinuous and $\C^*$-increasing.
\end{definition}

The authors coined this term in \cite[Definition~2.7]{chen2020fenchel}.
Examples of cones with the Fenchel--Moreau property include $\R^d_+$ (see \cite[Theorem~12.4]{rockafellar1970convex}), the set of positive semi-definite matrices in Example~\ref{example} (see \cite[Proposition~B.1]{chen2020hamiltonTensor}), finite-dimensional cones from mean-field spin glasses (see \cite[Proposition~5.1]{chen2022hamilton}), and more generally cones from a class called \textit{perfect cones} (the first two examples belong to this class; see \cite[Corollary~2.3]{chen2020fenchel}).

\subsection{Hopf--Lax Formula}
A classical reference to the Hopf--Lax formula is \cite{evans2010partial}.
\begin{proposition}[Hopf--Lax formula]\label{p.hopf-lax}
Suppose that $\C$ has the Fenchel--Moreau property. If $\psi \in \Gamma^\nearrow_\mathrm{Lip}(\C)$ and $\H:\cH\to\R$ satisfies that $\H\lfloor_\C\in \Gamma^\nearrow_\mathrm{locLip}(\C)$ and that $\H\lfloor_\C$ is convex and bounded below, then $f$ given by
\begin{align}\label{e.hopf_lax}
    f(t,x) = \sup_{y\in \C}\left\{\psi(x+y) - t \H^*\left(\frac{y}{t}\right)\right\},\quad\forall (t,x)\in\R_+\times\C.
\end{align}
is a Lipschitz viscosity solution of $\HJ(\mathring\C,\H;\psi)$ in $\mathfrak{M}\cap\mathfrak{L}_\mathrm{Lip}$.
\end{proposition}

\begin{proof}
To make sense of~\eqref{e.hopf_lax} at $t=0$, we rewrite the right-hand side of~\eqref{e.hopf_lax} as
\begin{align*}
    f(t,x) = \sup_{y\in \C}\inf_{z\in\C}\left\{\psi(x+y)-\la z, y\ra + t\H(z)\right\},\quad\forall (t,x)\in\R_+\times\C.
\end{align*}
Then, we can see that, when $t=0$, the supremum in this display must be achieved at $y=0$, implying $f(0,x) = \psi(x)$ for all $x\in\C$. Since $\psi$ is $\C^*$-increasing, it is straightforward that $f\in \mathfrak{M}$. Once we show that $f$ is Lipschitz, it is immediate that $f\in\mathfrak{L}_\mathrm{Lip}$. It remains to show that $f$ is a Lipschitz solution of $\HJ(\mathring\C,\H)$. We proceed in steps.

\bigskip

\textit{Step 1 (semigroup property).}
We show that for all $t> s\geq 0$,
\begin{align}\label{e.semigroup_hopf_lax}
    f(t,x) = \sup_{y\in \C}\left\{f(s,x+y)-  (t-s) \H^*\left(\frac{y}{t-s}\right) \right\}, \quad\forall x \in \C.
\end{align}
The convexity of $\H^*$ implies that
\begin{align*}
    \H^*\left(\frac{y+z}{t}\right)\leq \frac{s}{t}\H^*\left(\frac{y}{s}\right)+ \frac{t-s}{t}\H^*\left(\frac{z}{t-s}\right),\quad\forall y,z\in \C,
\end{align*}
which along with~\eqref{e.hopf_lax} yields that
\begin{align*}
    f(t,x)\geq \sup_{y,z\in \C}\left\{\psi(x+y+z) -s\H^*\left(\frac{y}{s}\right)- (t-s) \H^*\left(\frac{z}{t-s}\right)\right\}\\
    =\sup_{z\in \C}\left\{f(s,x+z)-(t-s)\H^*\left(\frac{z}{t-s}\right)\right\}.
\end{align*}

To show the converse inequality, we claim that for any fixed $(t,x)\in(0,\infty)\times\C$, there is $y\in\C$ satisfying
\begin{align}\label{e.maximizer}
    f(t,x) = \psi(x+y)- t\H^*\left(\frac{y}{t}\right).
\end{align}
Assuming this, we set $z=\frac{t-s}{t}y$ which satisfies $\frac{z}{t-s}=\frac{y-z}{s}=\frac{y}{t}$. By this,~\eqref{e.hopf_lax}, and~\eqref{e.maximizer}, we have
\begin{align*}
    f(s,x+z)-(t-s)\H^*\left(\frac{z}{t-s}\right)\geq \psi(x+z+y-z)-s \H^*\left(\frac{y-z}{s}\right) - (t-s)\H^*\left(\frac{z}{t-s}\right)
    \\
    =\psi(x+y)-t\H^*\left(\frac{y}{t}\right)=f(t,x),
\end{align*}
which yields the desired inequality.

It remains to verify the existence of $y$ in~\eqref{e.maximizer}. Fix any $\lambda>0$ and set $x=\lambda\frac{y}{|y| }$ in~\eqref{e.def_u*} for $\H^*$ to see that
\begin{align*}
    \H^*(y)\geq \lambda|y| -\sup_{\C\cap B(0,\lambda)}|\H|,\quad\forall y \in \C.
\end{align*}
Since $\H$ is locally Lipschitz, the supremum on the right is finite. Hence, we can deduce that
\begin{align}\label{e.liminf_H^*(y)/y}
    \liminf_{\substack{y\to\infty \\ y\in\C}}\frac{\H^*(y)}{|y| }=\infty.
\end{align}
We set $L=\|\psi\|_\mathrm{Lip}$. Then, the above implies the existence of $R>0$ such that $\H^*(\frac{y}{t})\geq (L+1)\frac{|y| }{t}$ for all $y$ satisfying $\frac{|y| }{t}> R$. These imply that
\begin{align*}
    \psi(x+y)-t\H^*\left(\frac{y}{t}\right)\leq\psi(x)+L|y| -(L+1)|y| = \psi(x)-|y|,
\end{align*}
for all $y$ satisfying $|y|>tR$. Therefore, the supremum in~\eqref{e.hopf_lax} can be taken over a bounded set. Also note that the function $y\mapsto\psi(x+y) - t\H^*(\frac{y}{t})$ is upper semi-continuous and locally bounded from above due to $\H^*(z)\geq -\H(0)$. Since $\cH$ is finite-dimensional, the maximizer must exist, which ensures the existence of $y$ in~~\eqref{e.maximizer} and thus completes the proof of~\eqref{e.semigroup_hopf_lax}.

\bigskip

\textit{Step 2 (Lipschitzness).}
We first show the following claim: for every $(t,x)\in(0,\infty)\times\C$, there is $y\in\C$ such that
\begin{align}\label{e.f(t,z)-f(t,x)}
    f(t,x)- f(t,x')\leq \psi(x+y)-\psi(x'+y),\quad\forall x'\in\C.
\end{align}
Fix any $(t,x)\in(0,\infty)\times\C$. Arguing as before, we can find $y\in\C$ such that~\eqref{e.maximizer} holds.
The Hopf--Lax formula~\eqref{e.hopf_lax} gives the lower bound
\begin{align*}
    f(t,x')\geq\psi(x'+y)-t\H^*\left(\frac{y}{t}\right),
\end{align*}
which along with~\eqref{e.maximizer} yields~\eqref{e.f(t,z)-f(t,x)}.

Now, for any $(t,x,x')\in(0,\infty)\times\C\times\C$, we apply~\eqref{e.f(t,z)-f(t,x)} to both $x$ and $x'$ to see that there exist $y,y'\in\C$ such that
\begin{align*}
    \psi(x+y')-\psi(x'+y')\leq f(t,x)-f(t,x')\leq \psi(x+y)-\psi(x'+y),
\end{align*}
which immediately implies that
\begin{align}\label{e.unif_lip_hopf_lax}
    \sup_{t>0}\|f(t,\cdot)\|_\mathrm{Lip}\leq \|\psi\|_\mathrm{Lip}.
\end{align}

Then, we show that
\begin{align}\label{e.f_t_lip_hopf_lax}
    \sup_{x\in\C}\|f(\cdot,x)\|_\mathrm{Lip} \leq |\H^*(0)|\vee\sup_{\C\cap B(0,\|\psi\|_\mathrm{Lip})}\left|\H\right|.
\end{align}
Let us fix any $x\in\C$ and $t> s\geq 0$. Then,~\eqref{e.semigroup_hopf_lax} yields
\begin{align*}
    f(t,x)\geq f(s,x)-(t-s)\H^*(0).
\end{align*}
where $\H^*(0)$ is finite by the assumption that $\H\lfloor_\C$ is bounded below.
Next, using~\eqref{e.unif_lip_hopf_lax}, we can obtain from~\eqref{e.semigroup_hopf_lax} that
\begin{align*}
    f(t,x)\leq f(s,x)+\sup_{y\in \C}\left\{\|\psi\|_\mathrm{Lip}|y| -(t-s)\H^*\left(\frac{y}{t-s}\right)\right\}.
\end{align*}
Changing the variable $\frac{y}{t-s}$ to $z$, and using $\|\psi\|_\mathrm{Lip}|z|=\la z, \frac{\|\psi\|_\mathrm{Lip}z}{|z| }\ra $, we can bound the supremum on the right-hand side of the above display by
\begin{align*}
    (t-s)\sup_{z\in \C}\left\{\|\psi\|_\mathrm{Lip}|z| -\H^*(z)\right\}
    \leq (t-s)\sup_{p\in\C\cap B(0,\|\psi\|_\mathrm{Lip})}\sup_{z\in \C}\left\{\la z, p\ra -\H^*(z)\right\}
    \\
    = (t-s) \sup_{\C\cap B(0,\|\psi\|_\mathrm{Lip})} \H,
\end{align*}
where the last equality follows from the Fenchel--Moreau property of $\C$. The above three displays together yield~\eqref{e.f_t_lip_hopf_lax}.

\bigskip

\textit{Step 3 (supersolution).}
Suppose that $f-\phi$ achieves a local minimum at $(t,x)\in (0,\infty)\times \mathring\C$ for some smooth function $\phi$. Then, 
\begin{align*}
    f(t-s,x+sy) - \phi(t-s,x+sy)\geq f(t,x)-\phi(t,x)
\end{align*}
for every $y\in \C$ and sufficiently small $s>0$. On the other hand,~\eqref{e.semigroup_hopf_lax} implies that
\begin{align*}
    f(t,x)\geq f(t-s,x+sy)-s\H^*(y).
\end{align*}
Combining the above two displays, we obtain that
\begin{align*}
    \phi(t,x)- \phi(t-s,x+sy)+s\H^*(y)\geq 0.
\end{align*}
Sending $s\to0$, we have that
\begin{align*}
    \partial_t\phi(t,x)-\la y,\nabla \phi(t,x)\ra +\H^*(y)\geq 0.
\end{align*}
Taking infimum over $y\in\C$ and using the Fenchel--Moreau property of $\C$, we obtain
\begin{align*}
    \left(\partial_t\phi-\H(\nabla\phi)\right)(t,x)\geq 0,
\end{align*}
which implies that $f$ is a supersolution of $\HJ(\mathring\C,\H)$.

\bigskip

\textit{Step 4 (subsolution).}
Suppose that $f-\phi$ achieves a local maximum at $(t,x)\in(0,\infty)\times \mathring\C$. 
Then, for every $y\in\C^*$ and $\epsilon>0$ small enough, we have
\begin{align*}
    \phi(t,x+\epsilon y)-\phi(t,x)\geq  f(t,x+\epsilon y)-f(t,x)\geq 0,
\end{align*}
since $f(t,\cdot)$ is $\C^*$-increasing due to $f\in\mathfrak{M}$. By taking $\epsilon$ to $0$, the display above implies that
\begin{align*}
    \la \nabla \phi(t,x), y\ra\geq 0.
\end{align*}
Therefore, we must have $\nabla \phi(t,x)\in\C$. Using again the maximality at $(t,x)$ and the smoothness of $\phi$, we have
\begin{align*}
    f(t',x') -f(t,y) \leq \partial_t\phi(t,x)(t'-t) +\la \nabla\phi(t,x),x'-x\ra + O(|t'-t|^2+|x'-x|^2)
\end{align*}
for $(t',x')$ sufficiently close to $(t,x)$. Setting
\begin{align*}
    \bar\phi(t',x') = \partial_t\phi(t,x)(t'-t) +\la \nabla\phi(t,x),x'-x\ra + C(|t'-t|^2+|x'-x|^2),\quad\forall (t',x')\in(0,\infty)\times \mathring\C,
\end{align*}
for some sufficiently large constant $C>0$, we have that $f-\bar\phi$ achieves a local maximum at $(t,x)$. Also,
\begin{align}\label{e.nabla_bar_phi}
    \nabla\bar\phi(t',x') = \nabla \phi(t,x) + 2C(x'-x)\in\C,\quad\forall (t',x')\in (0,\infty)\times (x+\C).
\end{align}

We want to show that
\begin{align}\label{e.subsol_contrad}
    \left(\partial_t\phi-\H(\nabla\phi)\right)(t,x)\leq 0.
\end{align}
Since the first-order derivatives of $\phi$ and $\bar\phi$ coincide at $(t,x)$,
we argue by contradiction and assume that there is $\delta>0$ such that
\begin{align*}
    \left(\partial_t\bar\phi-\H(\nabla\bar\phi)\right)(t',x')\geq \delta>0,
\end{align*}
for $(t',x')\in (0,\infty)\times (x+\C)$ sufficiently close to $(t,x)$, where we used \eqref{e.nabla_bar_phi}, the continuity of $\nabla\bar\phi$, and the continuity of $\H\lfloor_\C$ (we need to modify $\phi$ into $\bar\phi$ to ensure \eqref{e.nabla_bar_phi} because $\H$ is only assumed to be continuous on $\C$). The definition of $\H^*$ (in~\eqref{e.def_u*}) implies that
\begin{align}\label{e.der_phi(t',x')>delta}
    \partial_t\bar\phi(t',x') - \la q, \nabla \bar\phi(t',x')\ra + \H^*(q)\geq \delta
\end{align}
for all such $(t',x')$ and all $q\in\cH$.

To proceed, we show that there is $R>1$ such that for every $s>0$ sufficiently small there is $x_s\in \C$ such that
\begin{gather}
    f(t,x)=f(t-s,x+x_s) - s\H^*\left(\frac{x_s}{s}\right),\label{e.f(t,x)=f(t-s,x_s)...}
    \\
    |x_s|\leq Rs. \label{e.|x-x_s|}
\end{gather}
In view of~\eqref{e.unif_lip_hopf_lax} and~\eqref{e.f_t_lip_hopf_lax}, we set $L=\|f\|_\mathrm{Lip}<\infty$. By~\eqref{e.liminf_H^*(y)/y}, we can choose $R> 1$ to satisfy $\H^*(z)\geq 2L|z|$ for every $z\in\C$ satisfying $|z|>R$. Then, for every $y\in \C$ satisfying $\frac{|y| }{s}>R$, we have
\begin{align*}
    f(t-s,x+y)-s\H^*\left(\frac{y}{s}\right)\leq f(t,x)+Ls+L|y| -2L|y| <f(t,x)+Ls(1-R).
\end{align*}
Hence, the supremum in~\eqref{e.semigroup_hopf_lax} (with $s$ therein replaced by $t-s$) can be taken over $\{y\in \C:|y| \leq Rs\}$. Since $\cH$ is finite-dimensional, we can thus conclude the existence of $x_s\in \C$ satisfying~\eqref{e.f(t,x)=f(t-s,x_s)...} and~\eqref{e.|x-x_s|}.

Returning to the proof, we can compute that, for sufficiently small $s>0$,
\begin{align*}
    \bar\phi(t,x)- \bar\phi(t-s,x+x_s)& =\int_0^1\frac{\d}{\d r}\bar\phi(t+(r-1)s,x+(1-r)x_s)\d r\\
    & = \int_0^1\left(s\partial_t\bar\phi - \la x_s,\nabla\bar\phi\ra\right)(t+(r-1)s,x+(1-r)x_s)\d r.
\end{align*}
Using~\eqref{e.der_phi(t',x')>delta} with $q$ replaced by $\frac{x_s}{s}$, and~\eqref{e.f(t,x)=f(t-s,x_s)...}, we obtain from the above that
\begin{align*}
    \bar\phi(t,x)- \bar\phi(t-s,x+x_s) \geq s\delta - s\H^*\left(\frac{x_s}{s}\right)\geq s\delta+f(t,x)-f(t-s,x+x_s).
\end{align*}
Rearranging terms, we arrive at that, for all $s>0$ sufficiently small, 
\begin{align*}
    f(t-s,x+x_s)- \bar\phi(t-s,x+x_s) \geq s\delta +f(t,x)-\bar\phi(t,x),
\end{align*}
contradicting the local maximality of $f-\bar\phi$ at $(t,x)$. Hence,~\eqref{e.subsol_contrad} must hold, implying that $f$ is a subsolution.
\end{proof}

\subsection{Hopf Formula}
Classical references to the Hopf formula include \cite{bardi1984hopf,lions1986hopf}.
\begin{proposition}[Hopf formula]\label{p.hopf}
Suppose that $\C$ has the Fenchel--Moreau property. If $\psi\in \Gamma^\nearrow_\mathrm{Lip}(\C)$ is convex, and $\H:\cH\to \R$ satisfies that $\H\lfloor_\C$ is continuous, then $f$ given by
\begin{align}\label{e.hopf}
    f(t,x) = \sup_{z\in\C}\inf_{y\in\C} \left\{\la z,x-y\ra+\psi(y)+t\H(z)\right\},\quad\forall (t,x)\in\R_+\times\C,
\end{align}
is a Lipschitz viscosity solution of $\HJ(\mathring\C,\H;\psi)$ in $\mathfrak{M}\cap\mathfrak{L}_\mathrm{Lip}$.
\end{proposition}

\begin{proof}
Recall the definition of $g^*$ in \eqref{e.def_u*}.
It is easy to see that
\begin{align}\label{e.u**leq_u}
    g^{**}(x)\leq g(x),\quad\forall x \in \C.
\end{align}
We can rewrite~\eqref{e.hopf} as
\begin{align}
    f(t,x) 
    & = \sup_{z\in \C}\{\la z, x\ra_\cH - \psi^*(z)+t\H(z)\}\label{e.f_Hopf_1}\\
    & = (\psi^*-t\H)^*(x).\label{e.f_Hopf_2}
\end{align}

\textit{Step 1 (initial condition and monotonicity).}
Using~\eqref{e.f_Hopf_2}, we have $f(0,\cdot)=\psi^{**}$. Then, the Fenchel--Moreau property of $\C$ ensures $\psi^{**}=\psi$.
Since the supremum in the definition of $f$ is taken over $z\in\C$, it is clear that $f\in\mathfrak{M}$.

\textit{Step 2 (semigroup property).}
We want to show, for all $s\geq 0$,
\begin{align}\label{e.semigroup_hopf}
    f(t+s,\cdot)= \left(f^*(t,\cdot)-s\H\right)^*.
\end{align}
In view of the Hopf formula~\eqref{e.f_Hopf_2}, this is equivalent to
\begin{align}\label{e.semigroup_property}
    \left(\psi^* - (t+s)\H\right)^* = \left((\psi^*-t\H)^{**}-s\H\right)^*.
\end{align}
Since the conjugate is order-reversing (which means that $g_1^*\geq g_2^*$ if $g_1\leq g_2$),~\eqref{e.u**leq_u} implies
\begin{align}\label{e.semigruop_one_inequality}
    \left((\psi^*-t\H)^{**}-s\H\right)^*\geq \left(\psi^* - (t+s)\H\right)^* .
\end{align}

\smallskip

To see the other direction, we use~\eqref{e.u**leq_u} to get
\begin{align*}
    \frac{s}{t+s}\psi^* + \frac{t}{t+s}\left(\psi^*-(t+s)\H\right)^{**}\leq \psi^*-t\H.
\end{align*}
Notice that the left-hand side is convex, lower semicontinuous, and $\C^*$-increasing. Taking the biconjugate in the above display and applying the Fenchel--Moreau property of $\C$, we have
\begin{align*}
    \frac{s}{t+s}\psi^* + \frac{t}{t+s}\left(\psi^*-(t+s)\H\right)^{**}\leq (\psi^*-t\H)^{**}.
\end{align*}
Then, we rearrange terms and use~\eqref{e.u**leq_u} to see
\begin{align*}
    \left(\psi^*-(t+s)\H\right)^{**} - (\psi^*-t\H)^{**}\leq \frac{s}{t}\left((\psi^*-t\H)^{**}-\psi^*\right)\leq -s\H,
\end{align*}
and thus
\begin{align*}
    \left(\psi^*-(t+s)\H\right)^{**}\leq (\psi^*-t\H)^{**}-s\H.
\end{align*}
Taking the conjugate on both sides, using its order-reversing property, and invoking the Fenchel--Moreau property of $\C$, we get
\begin{align*}
     \left(\psi^* - (t+s)\H\right)^*\geq  \left((\psi^*-t\H)^{**}-s\H\right)^*,
\end{align*}
which together with~\eqref{e.semigruop_one_inequality} verifies~\eqref{e.semigroup_property}.

\textit{Step 3 (Lipschitzness).}
Since $\psi$ is Lipschitz, we have $\psi^*(z)=\infty$ outside the compact set $B=\{z\in\C:|z|_\cH \leq \|\psi\|_\mathrm{Lip}\}$. This together with~\eqref{e.f_Hopf_1} implies that for each $x\in\C$, there is $z\in B$ such that
\begin{align}\label{e.f(t,x)=<z,x>...}
    f(t,x) = \la z, x\ra_\cH  - \psi^*(z)+t\H(z).
\end{align}
Using this and~\eqref{e.f_Hopf_1}, we get that
\begin{align*}
    f(t,x) -f(t,x')\leq \la z, x-x'\ra_\cH \leq \|\psi\|_\mathrm{Lip}|x-x'|_\cH,\quad\forall x'\in\C.
\end{align*}
By symmetry, we conclude that $f(t,\cdot)$ is Lipschitz, and the Lipschitz coefficient is uniform in $t$.
To show the Lipschitzness in $t$, we fix any $x\in\C$. Then, we have, for some $z\in B$,
\begin{align*}
    f(t,x)&=\la z, x\ra_\cH-\psi^*(z)+t\H(z)\leq f(t',x) + (t-t')\H(z)\\
    &\leq f(t',x) +|t'-t|\left(\sup_{|z|_\cH \leq \|\psi\|_\mathrm{Lip}}|\H(z)|\right).
\end{align*}
Again by symmetry, the Lipschitzness of $f(\cdot,x)$ is obtained, and its coefficient is independent of $x$. 
In particular, we get $f \in \mathfrak{L}_\mathrm{Lip}$.

\textit{Step 4 (subsolution).}
Let $\phi:(0,\infty)\times \C\to\R$ be smooth. Suppose that $f-\phi$ achieves a local maximum at $(t,x)\in (0,\infty)\times \mathring\C$. Arguing as above, there is $z\in\C$ such that~\eqref{e.f(t,x)=<z,x>...} holds. By this and~\eqref{e.f_Hopf_1}, we have, for $s\in[ 0,t]$ and small $h\in \C$,
\begin{align*}
    f(t,x) \leq f(t-s,x+h)- \la z, h\ra_\cH+s\H( z).
\end{align*}
The local maximality of $f-\phi$ at $(t,x)$ gives
\begin{align*}
    f(t-s,x+h)-\phi(t-s,x+h)\leq f(t,x)-\phi(t,x).
\end{align*}
for small $s\in[0,t]$ and small $h\in\cH$. Then, we combine the above two inequalities to get
\begin{align}\label{e.phi(t,x)_upper_bound_verify_HJ}
    \phi(t,x)-\phi(t-s,x+h) \leq - \la z, h\ra_\cH+s\H( z),
\end{align}
for sufficiently small $s\geq0$ and $h\in\cH$.
We can set $s=0$, substitute $\eps y$ for $ h$ for any $y\in\cH$ and sufficiently small $\eps>0$, and then send $\eps\to 0$ to see $\la y, \nabla \phi(t,x) -  z \ra_\cH\geq 0$ for all $y\in\cH$, which implies $\nabla \phi(t,x) =  z$.
Then, we set $h=0$ in~\eqref{e.phi(t,x)_upper_bound_verify_HJ} and take $s\to 0$ to obtain $\partial_t \phi(t,x)\leq\H(z)$.
Hence, we get $\partial_t \phi(t,x)-\H(\nabla\phi(t,x))\leq 0$ and thus $f$ is a viscosity subsolution.

\textit{Step 6 (supersolution).}
Let $(t,x)\in(0,\infty)\times \mathring\C$ be a local minimum point for $f-\phi$. Due to~\eqref{e.f_Hopf_1}, $f$ is convex in both variables. Since $\C$ is also convex, we have, for all $(t',x')\in (0,\infty)\times \C$ and all $\lambda \in (0,1]$,
\begin{align*}
    f(t',x')-f(t,x)\geq \frac{1}{\lambda}\left(f\left(t+\lambda(t'-t),x+\lambda(x'-x)\right) - f(t,x)\right).
\end{align*}
The local minimality of $f-\phi$ at $(t,x)$ implies
\begin{align*}
    f\left(t+\lambda(t'-t),x+\lambda(x'-x)\right) - f(t,x)\geq \phi\left(t+\lambda(t'-t),x+\lambda(x'-x)\right) - \phi(t,x).
\end{align*}
Using the above two displays and sending $\lambda\to0$, we obtain
\begin{align}\label{e.sup_sol_supp_0}
    f(t',x')-f(t,x)\geq r(t'-t) + \la a,\,x'-x\ra_\cH,\quad \forall (t',x')\in\R_+\times\C,
\end{align}
where, for convenience, we set $r=\partial_t\phi(t,x)$ and $a = \nabla\phi(t,x)$.
Setting $t'=t$ in~\eqref{e.sup_sol_supp_0} and using $f\in\mathfrak{M}$ and $x\in\mathring\C$, we can verify $a\in \C$.

Fix some $s\in(0,t)$, we set
\begin{align*}
    \eta(x')=f(t,x)-rs + \la a,x'-x\ra_\cH,\quad\forall x'\in\C.
\end{align*}
Setting $t'=t-s$ in~\eqref{e.sup_sol_supp_0}, we have
\begin{align*}
    f(t-s,x')\geq \eta(x'),\quad\forall x'\in\C.
\end{align*}
Applying the order-reversing property of the conjugate twice, we obtain
\begin{align*}
    \left(f^*(t-s,\cdot)-s\H\right)^*\geq  \left(\eta^*-s\H\right)^*.
\end{align*}
By the semigroup property~\eqref{e.semigroup_hopf}, we have
\begin{align*}
    f(t,\cdot)\geq \left(\eta^*-s\H\right)^*.
\end{align*}
By $a\in\C$ and the definitions of the conjugate in~\eqref{e.def_u*}, the above yields
\begin{align*}
    f(t,x)\geq \la a, x\ra_\cH-\eta^*(a)+s\H(a).
\end{align*}
On the other hand, using the definition of $\eta$, we can compute
\begin{align*}
    \eta^*(a)= -f(t,x)+rs+\la a, x\ra_\cH.
\end{align*}
The above two displays along with the definition of $r$ and $a$ yield $\left(\partial_t\phi-\H(\nabla\phi(t,x))\right)(t,x)\geq 0$, which verifies that $f$ is a supersolution.
\end{proof}

\appendix
\section{Additional results}

In anticipation of future needs, we record variations of two of the results stated above.

\subsection{A comparison principle on the closed cone}

The following proposition is a variation of Proposition~\ref{p.cp_H}. The result below concerns $\HJ(\C,\sF)$ instead of $\HJ(\mathring\C,\sF)$ and does not require $\sF$ to be $\C^*$-increasing.

\begin{proposition}[Comparison principle]Assume that one of the following holds:
\begin{enumerate}[label=\rm{(\roman*)}]
    \item \label{i.p.cp_F_case1} $\sF\in\Gamma_\mathrm{Lip}(\cH)$, and $u,v\in\mathfrak{L}$;

    \item \label{i.p.cp_F_case2} $\sF\in\Gamma_\mathrm{locLip}(\cH)$, and $u,v\in\mathfrak{L}_\mathrm{Lip}$;
\end{enumerate}
If $u,v$ are a subsolution and a supersolution of $\HJ(\C,\sF)$, respectively, then $\sup_{\R_+\times\C}(u-v)=\sup_{\{0\}\times \C}(u-v)$.
\end{proposition}

\begin{proof}
Setting $C_1 = \sup_{\{0\}\times \C}(u-v)$, we can assume that $C_1$ is finite, otherwise there is nothing to show.
We argue by contradiction and assume that $\sup_{\R_+\times\C}(u-v)>\sup_{\{0\}\times\C}(u-v)$. Then, we can fix $T>0$ sufficiently large so that
\begin{align}\label{e.u-v_closed}
    \sup_{[0,T)\times\C}(u-v)>\sup_{\{0\}\times\C}(u-v).
\end{align}
Using $u,v\in\mathfrak{L}$, we can fix a constant $L>1$ such that
\begin{gather*}
    L>\|u(0,\cdot)\|_\mathrm{Lip}\vee \|v(0,\cdot)\|_\mathrm{Lip},
    \\
    |u(t,x)-u(0,x)|\vee |v(t,x)-v(0,x)|\leq Lt, \quad\forall (t,x)\in\R_+\times\C.
\end{gather*}
If necessary, we make $L$ larger to satisfy
\begin{align*}
    L > \sup_{t\in\R_+}\|u(t,\cdot)\|_\mathrm{Lip}\vee\|v(t,\cdot)\|_\mathrm{Lip},\quad\text{in case~\ref{i.p.cp_F_case2}}.
\end{align*}
We set
\begin{align*}
    V=
    \begin{cases}
    \|\sF\|_\mathrm{Lip},\quad &\text{in case~\ref{i.p.cp_F_case1}},
    \\
    \|\sF\lfloor_{B(0,2L+1)}\|_\mathrm{Lip}, &\text{in case~\ref{i.p.cp_F_case2}},
    \end{cases}
\end{align*}
Let $\theta:\R\to\R_+$ satisfy~\eqref{e.theta}.
For $R>0$ to be chosen, we define
\begin{align*}
    \chi(t,x) = \theta\left(\left(1+|x|^2\right)^\frac{1}{2}+Vt-R\right),\quad\forall (t,x)\in \R_+\times \C.
\end{align*}
It is immediate that
\begin{gather}
    \sup_{(t,x)\in\R_+\times\C}|\nabla\chi(t,x)|\leq 1,\label{e.cp2.|nabla_Phi|_closed}
    \\
    \partial_t\chi\geq V|\nabla \chi|,\label{e.cp2.d_tPhi>_closed}
    \\
    \chi(t,x)\geq |x|-R, \quad\forall (t,x)\in\R_+\times\C. \label{e.cp2.Phi(t,x)>M|x|_closed}
\end{gather}

For $\delta>0$ to be determined, we define
\begin{align*}
    \zeta(t,t',x)=\chi(t,x) + \delta t+\frac{\delta}{T-t}+\frac{\delta}{T-t'},\quad \forall (t,t',x) \in [0,T)^2\times\C.
\end{align*}
In view of \eqref{e.u-v_closed}, we fix $\delta>0$ sufficiently small and $R$ sufficiently large so that
\begin{align}\label{e.cp2.u-v-chi_closed}
    \sup_{(t,x)\in[0,T)\times\C}(u(t,x)-v(t,x)-\zeta(t,t,x))>\sup_{x\in\C}(u(0,x)-v(0,x)-\zeta(0,0,x)).
\end{align}
For each $\alpha>1$, we introduce
\begin{align*}
    \Psi_\alpha(t,t',x,x')= u(t,x)-v(t',x')-\frac{\alpha}{2}(|t-t'|^2+|x-x'|^2)-\zeta(t,t',x),
    \\
    \quad\forall (t,t',x,x')\in[0,T)^2\times\C^2. 
\end{align*}
By the semi-continuity of $u$ and $v$, we have that $\Psi_\alpha$ is upper semicontinuous.
By the definition of $L$ and $C_1$, we can see that, for all $t,t'\in[0,T)$ and $x,x'\in\C$,
\begin{align*}
    u(t,x)-v(t',x')\leq 2LT+u(0,x)-v(0,x)+L|x-x'|\leq 2LT+C_1 + L|x-x'|,
\end{align*}
which along with $\alpha>1$ and \eqref{e.cp2.Phi(t,x)>M|x|_closed} implies that
\begin{align*}
    \Psi_\alpha(t,t',x,x')\leq 2LT +C_1+L|x-x'|-\frac{1}{2}|x-x'|^2-(|x|-R).
\end{align*}
Hence, $\Psi_\alpha$ is bounded from above uniformly in $\alpha>1$ and decays as $|x|,|x'|\to\infty$, which implies that $\Psi_\alpha$ achieves its supremum at some $(t_\alpha,t'_\alpha,x_\alpha,x'_\alpha)$, and there is a constant $C_2>0$ such that
\begin{align*}
    |x_\alpha|,\ |x'_\alpha|\leq C_2,\quad\forall \alpha>1.
\end{align*}
Setting $C_0=\Psi_\alpha(0,0,0,0)$ which is independent of $\alpha$, we have
\begin{align*}
    C_0\leq \Psi_\alpha(t_\alpha,t'_\alpha,x_\alpha,x'_\alpha)\leq 2LT+C_1+2LC_2-\frac{\alpha}{2}(|t_\alpha-t'_\alpha|^2+|x_\alpha-x'_\alpha|^2).
\end{align*}
From this, we can see that $\alpha(|t_\alpha-t'_\alpha|^2+|x_\alpha-x'_\alpha|^2)$ is bounded as $\alpha\to\infty$.
Hence, passing to a subsequence if necessary, we may assume $t_\alpha,t'_\alpha\to t_0$ and $x_\alpha,x'_\alpha\to x_0$ for some $(t_0,x_0)\in [0,T]\times \C$.

Then, we show $t_0\in(0,T)$.
Since
\begin{align*}
    C_0\leq \Psi_\alpha(t_\alpha,t'_\alpha,x_\alpha,x'_\alpha)\leq 2LT+C_1+2LC_2-\frac{\delta}{T-t_\alpha},
\end{align*}
we must have that $t_\alpha$ is bounded away from $T$ uniformly in $\alpha$, which implies $t_0<T$. Since
\begin{align*}
    u(t_\alpha,x_\alpha) - v(t'_\alpha,x'_\alpha)-\zeta(t_\alpha,t'_\alpha,x_\alpha)\geq \Psi_\alpha(t_\alpha,t'_\alpha,x_\alpha,x'_\alpha) \\
    \geq \sup_{(t,x)\in[0,T)\times\C}(u(t,x)-v(t,x)-\zeta(t,t,x))\geq u(t_0,x_0)-v(t_0,x_0)-\zeta(t_0,t_0,x_0),
\end{align*}
sending $\alpha\to\infty$, we deduce that
\begin{align*}
    u(t_0,x_0)-v(t_0,x_0)-\zeta (t_0,t_0,x_0) = \sup_{(t,x)\in[0,T)\times\C}(u(t,x)-v(t,x)-\zeta(t,t,x)).
\end{align*}
This along with \eqref{e.cp2.u-v-chi_closed} implies that $t_0>0$. In conclusion, we have $t_0\in(0,T)$, and thus $t_\alpha,t'_\alpha\in(0,T)$ for sufficiently large $\alpha$. Henceforth, we fix any such $\alpha$.

Before proceeding, for case~\ref{i.p.cp_F_case2}, we need a bound on $|x_\alpha-x'_\alpha|$.
Using $\Psi_\alpha(t_\alpha,t'_\alpha,x_\alpha,x_\alpha)-\Psi_\alpha(t_\alpha,t'_\alpha,x_\alpha,x'_\alpha)\leq 0$, the computation that
\begin{align*}
    \Psi_\alpha(t_\alpha,t'_\alpha,x_\alpha,x_\alpha)-\Psi_\alpha(t_\alpha,t'_\alpha,x_\alpha,x'_\alpha) = v(t'_\alpha,x'_\alpha)-v(t'_\alpha,x_\alpha)+\frac{\alpha}{2}|x_\alpha-x'_\alpha|^2,
\end{align*}
and the definition of $L$, we can obtain
\begin{align}\label{e.cp2.|x_alpha-x'_alpha|_closed}
    \alpha|x_\alpha-x'_\alpha|\leq  2L,\quad\text{in case~\ref{i.p.cp_F_case2}}.
\end{align}

With this, we return to the proof. Since the function
\begin{align*}
    (t,x)\mapsto \Psi_\alpha(t,t'_\alpha,x,x'_\alpha)
\end{align*}
achieves its maximum at $(t_\alpha,x_\alpha)\in (0,T)\times \C$, by the assumption that $u$ is a subsolution, we have
\begin{align}\label{e.cp2.u_sub_Phi_alpha_closed}
    \alpha(t_\alpha-t'_\alpha) + \delta +\delta(T-t_\alpha)^{-2}+\partial_t\chi(t_\alpha,x_\alpha)-\sF\left(\alpha(x_\alpha - x'_\alpha)+\nabla\chi(t_\alpha,x_\alpha)\right)\leq 0
\end{align}
On the other hand, since the function
\begin{align*}
    (t',x')\mapsto \Psi_\alpha(t_\alpha,t',x_\alpha,x')
\end{align*}
achieves its minimum at $(t'_\alpha,x'_\alpha)\in (0,T)\times \C$, by the assumption that $v$ is a subsolution, we have
\begin{align}\label{e.cp2.v_super_Phi_alpha_closed}
    \alpha(t_\alpha-t'_\alpha) -\delta(T-t'_\alpha)^{-2} -\sF\left(\alpha(x_\alpha - x'_\alpha)\right)\geq 0.
\end{align}
In case~\ref{i.p.cp_F_case2}, by \eqref{e.cp2.|nabla_Phi|_closed} and \eqref{e.cp2.|x_alpha-x'_alpha|_closed}, the arguments inside $\sF$ in both \eqref{e.cp2.u_sub_Phi_alpha_closed} and \eqref{e.cp2.v_super_Phi_alpha_closed} have norms bounded by $2L+1$. Taking the difference of \eqref{e.cp2.u_sub_Phi_alpha_closed} and \eqref{e.cp2.v_super_Phi_alpha_closed}, and using the definition of $V$ and \eqref{e.cp2.d_tPhi>_closed}, we obtain that, in both cases,
\begin{align*}
    \delta\leq V|\nabla \chi(t_\alpha,x_\alpha)|-\partial_t\chi(t_\alpha,x_\alpha)\leq 0,
\end{align*}
contradicting the fact that $\delta>0$. Therefore, the desired result must hold.
\end{proof}

\subsection{Spatial monotonicity on the closed cone}

The following proposition is a variation of Proposition~\ref{p.monotone}. The result below concerns $\HJ(\C,\sF)$ instead of $\HJ(\mathring\C,\sF)$ and does not require $\sF$ to be $\C^*$-increasing.

\begin{proposition}[Spatial monotonicity]If $f\in\mathfrak{L}_\mathrm{Lip}$ solves $\HJ(\C,\sF;\psi)$ for some  $\psi\in\Gamma^\nearrow_\mathrm{Lip}(\C)$ and some $\sF\in\Gamma_\mathrm{locLip}(\cH)$, then $f\in\mathfrak{M}$.
\end{proposition}

\begin{remark}\label{r.d_C*}
We will need the function $d_{\C^*}$ defined below, which can be obtained from the function in \eqref{e.d} by replacing $\C$ by $\C^*$ (also recall that $(\C^*)^* =\C$),
\begin{align}\label{e.d*}
    d_{\C^*}(x) = \inf_{y\in\C,\ |y|=1}\la y, x\ra,\quad x \in \C^*.
\end{align}
For $d=d_{\C^*}$, Lemma~\ref{l.d} holds with $\C$ and $\C^*$ therein replaced by $\C^*$ and $\C$, respectively.
\end{remark}

\begin{proof}
We argue by contradiction and suppose that
\begin{align*}
    f(t_0,x_0)>f(t_0,x'_0)
\end{align*}
for some $t_0>0$ and $x_0,x'_0\in\C$ satisfying $x'_0-x_0\in\C^*$. We set $\Omega = \{(x,x')\in\C^2:x'-x\in\C^*\}$. Since $f(0,\cdot)$ is $\C^*$-nondecreasing, we thus have
\begin{align}\label{e.contrad_assump_1}
    \sup_{\substack{t\in\R_+ \\ (x,x')\in\Omega}}\left(f(t,x)-f(t,x')\right) > 0\geq  \sup_{(x,x')\in\Omega}\left(f(0,x)-f(0,x')\right).
\end{align}

\textit{Step~1.}
We introduce auxiliary functions and fix some constants.
We set $d=d_{\C^*}$ in \eqref{e.d*}.
Due to \eqref{e.contrad_assump_1}, there are $\delta>0$ and $T>0$ such that
\begin{align}
    \sup_{\substack{t\in[0,T) \\ (x,x')\in\Omega}}\left(f(t,x)-f(t,x')- \delta t -\zeta_1(t,t) -\frac{\delta}{d(x'-x)}-2\delta|x-x'|^2\right) \notag
    \\
    > \sup_{(x,x')\in\Omega}\left(f(0,x)-f(0,x')-\zeta_1(0,0)-\frac{\delta}{d(x'-x)}-2\delta|x-x'|^2\right). \label{e.contrad_assump_2}
\end{align}
where we have set
\begin{align*}
    \zeta_1(t,t') =  \delta(T-t)^{-1} + \delta(T-t')^{-1}.
\end{align*}
Fixing any $y_0\in \C\cap\C^*$ satisfying $y_0\not\in \partial \C^*$ (allowed by Lemma~\ref{l.cone_itr}), we set
\begin{align*}
    C_0 = -\left( f(0,0)-f(0,y_0)-\zeta_1(0,0) -\frac{\delta}{d(y_0)} -2 \delta|y_0|^2\right)
\end{align*}
which is finite.
Since $f\in\mathfrak{L}_\mathrm{Lip}$, $f(0,\cdot)=\psi$, and $\psi$ is Lipschitz, we can fix a constant $L>1$ to satisfy
\begin{gather}
    L > \|\psi\|_\mathrm{Lip}\vee\sup_{t>0}\|f(t,\cdot)\|_\mathrm{Lip},\label{e.Llowerbd}
    \\
    |f(t,x) - \psi(x)|\leq L t,\quad\forall (t,x)\in \R_+\times \C.\nonumber
\end{gather}
We set $K$ to be a positive constant depending only on $\delta, T,C_0,L$ and given explicitly by \eqref{e.K} and \eqref{e.C_2}.
Lastly, we set
\begin{gather}\label{e.V}
    V = \sup_{p,p'\in B(0,K+L+KL)}\frac{|\sF(p)-\sF(p')|}{|p-p'|}.
\end{gather}
Let $\theta:\R\to\R_+$ satisfy \eqref{e.theta}.
Then, we take
\begin{align*}
    \zeta_2(t,x) = \theta\left(\left(1+|x|^2\right)^\frac{1}{2}+Vt -R\right)
\end{align*}
where $R>0$ is chosen sufficiently large so that
\begin{align}\label{e.zeta_2(0,y_0)=0}
    \zeta_2(0,0) =0
\end{align}
and, due to \eqref{e.contrad_assump_2},
\begin{align}
    \sup_{\substack{t\in[0,T) \\ (x,x')\in\Omega}}\left(f(t,x)-f(t,x')- \delta t -\zeta_1(t,t)-\zeta_2(t,x) -\frac{\delta}{d(x'-x)}-2\delta|x-x'|^2\right) \notag
    \\
    > \sup_{(x,x')\in\Omega}\left(f(0,x)-f(0,x')-\zeta_1(0,0)-\zeta_2(0,x)-\frac{\delta}{d(x'-x)}-2\delta|x-x'|^2\right). \label{e.contrad_assump_3}
\end{align}
We record the following useful estimates: for all $t,t',x,x'$,
\begin{gather}
    f(t,x)- f(t',x') \leq Lt+ Lt' +L|x-x'|, \label{e.f-f<}
    \\
    \zeta_2(t,x) \geq |x|-R, \label{e.zeta_2>}
    \\
    \partial_t\zeta_2(t,x) \geq V|\nabla\zeta_2(t,x)|. \label{e.d_tzeta_2>}
\end{gather}

For $\alpha>1$ to be chosen, we define
\begin{align*}
    \Psi_\alpha(t,t',x,x',y) = f(t,x) - f(t',x') -\delta t -\zeta_1(t,t') - \zeta_2(t,x)-\zeta_3(x,x',y) - \alpha|t-t'|^2 - \delta|x-x'|^2,
    \\
    \forall (t,t',x,x',y)\in [0,T)^2\times \Omega\times \C^*,
\end{align*}
where $\zeta_3$ depends on $\alpha$ and is given by
\begin{align*}
    \zeta_3(x,x',y) = \alpha|x'-x-y|^2 +\frac{\delta}{d(y)}+\delta|y|^2,\quad\forall (x,x',y) \in \Omega\times \C^*.
\end{align*}
Due to \eqref{e.zeta_2(0,y_0)=0} and $y_0\in\C\cap\C^*$, note that
\begin{align}\label{e.Psi>C_0}
    \Psi_\alpha(0,0,0,y_0,y_0) = -C_0,\quad\forall \alpha.
\end{align}
In summary, we have fixed constants $\delta, T, y_0, C_0, L, K, V, R$ and introduced $\Psi_\alpha$ for $\alpha$ to be chosen.

\textit{Step~2.}
We show that there is a maximizer of $\Psi_\alpha$.
Temporarily, fix any $\alpha>1$. Let $(t_{\alpha,n},t'_{\alpha,n},x_{\alpha,n},x'_{\alpha,n},y_{\alpha,n})$ be a maximizing sequence of $\Psi_\alpha$. For sufficiently large $n$, we have $\Psi_\alpha(t_{\alpha,n},t'_{\alpha,n},x_{\alpha,n},x'_{\alpha,n},y_{\alpha,n})\geq \Psi_\alpha(0,0,0,y_0,y_0)-1$, which along with \eqref{e.f-f<}, \eqref{e.zeta_2>}, \eqref{e.Psi>C_0} and that $\zeta_3(x,x',y)\geq \delta|y|^2$ implies that
\begin{align*}
    Lt_{\alpha,n}+Lt'_{\alpha,n}+L|x_{\alpha,n}-x'_{\alpha,n}| - (|x_{\alpha,n}| - R) - \delta |y_{\alpha,n}|^2 - \delta |x_{\alpha,n}-x'_{\alpha,n}|^2   \geq - C_0-1
\end{align*}
Using $t_{\alpha,n},t'_{\alpha,n} \leq T$ and
\begin{align}\label{e.Lr-deltar^2}
    L r- \delta r^2  \leq \frac{L^2}{4\delta}\quad\forall r\in\R,
\end{align}
we can see that $x_{\alpha,n}$ and $y_{\alpha,n}$ are bounded uniformly in $\alpha, n$. The above display also implies that
\begin{align*}
    2LT + L|x_{\alpha,n}-x'_{\alpha,n}|+R + C_0+1 \geq \delta|x_{\alpha,n}-x'_{\alpha,n}|^2,
\end{align*}
which yields the boundedness of $x_{\alpha,n}-x'_{\alpha,n}$ and thus that of $x'_{\alpha ,n}$. Hence, we can pass $n$ to infinity along a subsequence to obtain a maximizer $(t_\alpha,t'_\alpha,x_\alpha,x'_\alpha,y_\alpha)$ of $\Psi_\alpha$. Moreover, there is a positive constant $C_1$ (depending only on $C_0,L,T,R,\delta$) such that
\begin{align}\label{e.<C_1}
    t_\alpha,\, t'_\alpha<T,\qquad |x_\alpha|,\,|x'_\alpha|,\,|y_\alpha|<C_1,\quad\forall \alpha>1,
\end{align}
where the strict upper bound on $ t_\alpha,\, t'_\alpha$ is due to the presence of $\zeta_1$ in $\Psi_\alpha$.

\textit{Step~3.}
We derive some estimates on $t_\alpha,t'_\alpha,x_\alpha,x'_\alpha,y_\alpha$.

Using $\Psi_\alpha(t_\alpha,t'_\alpha,x_\alpha,x'_\alpha,y_\alpha)\geq -C_0$ (due to \eqref{e.Psi>C_0}), we have
\begin{align*}
    \left(f(t_\alpha,x_\alpha)-f(t'_\alpha,x'_\alpha) - \delta|x_\alpha-x'_\alpha|^2\right) -\alpha|t_\alpha-t'_\alpha|^2 - \frac{\delta}{d(y_\alpha)} - \delta|y_\alpha|^2 \geq -C_0.
\end{align*}
Using \eqref{e.f-f<}, \eqref{e.Lr-deltar^2}, and $t_\alpha,t'_\alpha\leq T$, we can bound the term in parentheses in the above display by $2LT + \frac{L^2}{4\delta}$. Setting
\begin{align}\label{e.C_2}
    C_2 = \left(2LT + \frac{L^2}{4\delta} + C_0\right)\vee 1,
\end{align}
we have
\begin{align}
    |t_\alpha- t'_\alpha|\leq \alpha^{-\frac{1}{2}}C_2^{\frac{1}{2}}, \label{e.t-t'}
    \\
    d(y_\alpha)>\delta C_2^{-1},
    \label{e.d(y)>}
    \\
    |y_\alpha|\leq \delta^{-\frac{1}{2}}C_2^\frac{1}{2}. \label{e.y<}
\end{align}

By Lemma~\ref{l.d}~\eqref{i.d(x)=0} and Remark~\ref{r.d_C*} for $d=d_{\C^*}$, the lower bound in \eqref{e.d(y)>} implies that $y_\alpha$ is in the interior of $\C^*$. Since $y\mapsto \Psi_\alpha(t_\alpha,t'_\alpha,x_\alpha,x'_\alpha,y)$ achieves a local maximum at $y_\alpha$, the function $y\mapsto -\zeta_3(x_\alpha,x'_\alpha,y)$ has a local maximum at $y_\alpha$. Using Lemma~\ref{l.d}~\eqref{i.g-1/d} and Remark~\ref{r.d_C*}, we get that
\begin{align*}
   \left(d(y_\alpha)\right)^2\left(2 \delta^{-1}\alpha(y_\alpha-(x'_\alpha-x_\alpha)) + 2 y_\alpha\right)\in D^+d(y_\alpha),
\end{align*}
which, along with Lemma~\ref{l.d}~\eqref{i.Dd} (along with Remark~\ref{r.d_C*}), \eqref{e.d(y)>}, and \eqref{e.y<}, implies that
\begin{align}
    |2\alpha(x'_\alpha -x_\alpha - y_\alpha)| \leq  K, \label{e.<K}
\end{align}
where we have set
\begin{align}\label{e.K}
    K = \delta^{-1}C_2^2+ 2\delta^\frac{1}{2}C_2^\frac{1}{2}.
\end{align}
In particular, we have
\begin{align}
    |x'_\alpha -x_\alpha - y_\alpha| = O\left(\alpha^{-1}\right) \label{e.x'-x-y}
\end{align}

By passing to a subsequence, due to \eqref{e.<C_1}, and \eqref{e.t-t'}, and \eqref{e.x'-x-y}, we can assume that $(t_\alpha,t'_\alpha,x_\alpha,x'_\alpha,y_\alpha)$ converges to some $(t_\infty,t_\infty,x_\infty,x'_\infty,x'_\infty - x_\infty)$ as $\alpha\to \infty$. Since
\begin{align*}
    &\Psi_\alpha(t_\alpha,t'_\alpha,x_\alpha,x'_\alpha,y_\alpha)
    \\
    &\leq f(t_\alpha, x_\alpha) - f(t'_\alpha,x'_\alpha)-\delta t_\alpha -\zeta_1(t_\alpha,t'_\alpha) -\zeta_2(t_\alpha,x_\alpha)  - \frac{\delta}{d(y_\alpha)}-\delta|y_\alpha|^2-\delta|x_\alpha-x'_\alpha|^2,
\end{align*}
and, for all $t\in[0,T)$ and $(x,x')\in\Omega$,
\begin{align*}
    \Psi_\alpha(t_\alpha,t'_\alpha,x_\alpha,x'_\alpha,y_\alpha) &\geq \Psi_\alpha(t,t,x,x',x'-x) 
    \\
    &= f(t,x)-f(t,x') -\delta t -\zeta_1(t,t) -\zeta_2(t,x) - \frac{\delta}{d(x'-x)}-2\delta|x-x'|^2,
\end{align*}
we obtain by passing $\alpha\to\infty$ that
\begin{align*}
    f(t_\infty, x_\infty) - f(t_\infty,x'_\infty)-\delta t_\infty -\zeta_1(t_\infty,t_\infty) -\zeta_2(t_\infty,x_\infty)  - \frac{\delta}{d(x'_\infty- x_\infty)}- 2\delta|x_\infty-x'_\infty|^2
\end{align*}
is greater than or equal to the left-hand side of \eqref{e.contrad_assump_3}. Hence, we must have $t_\infty>0$ and thus
\begin{align}\label{e.t,t'>0}
    t_\alpha,\, t'_\alpha>0
\end{align}
for sufficiently large $\alpha$.

\textit{Step~4.}
We derive a contradiction by using the maximality of $\Psi_\alpha$ at $(t_\alpha,t'_\alpha,x_\alpha,x'_\alpha,y_\alpha)$ and the assumption that $f$ solves $\HJ(\C,\sF)$.

We infer from \eqref{e.d(y)>} and \eqref{e.x'-x-y} that $x'_\alpha-x_\alpha$ is in the interior of $\C^*$ for sufficiently large $\alpha$. Fix some $\alpha$ such that this and \eqref{e.t,t'>0} hold. Set
\begin{gather*}
    \phi(t,x) = f(t,x) - \Psi_\alpha(t,t'_\alpha,x,x_\alpha,y_\alpha),
    \\
    \phi'(t',x') = \Psi_\alpha(t_\alpha,t',x_\alpha,x',y_\alpha) + f(t',x').
\end{gather*}
Since $(t_\alpha,t'_\alpha,x_\alpha,x'_\alpha)$ maximizes $\Psi_\alpha$, by the choice of this $\alpha$ (ensuring $ t_\alpha,\, t'_\alpha\in(0,T)$ and $x'_\alpha-x_\alpha\in \itr \C^*$), we have that $f-\phi$ achieves a local maximum at $(t_\alpha,x_\alpha)\in(0,\infty)\times \C$ and $f-\phi'$ achieves a local minimum at $(t'_\alpha,x'_\alpha)\in(0,\infty)\times \C$.
Since $f$ solves $\HJ(\C,\sF)$, these imply that
\begin{gather}
    \partial_t\phi(t_\alpha,x_\alpha) - \sF\left(\nabla \phi(t_\alpha,x_\alpha)\right)\leq 0 \label{e.xi_sub}
    \\
    \partial_t\phi'(t'_\alpha,x'_\alpha) - \sF\left(\nabla \phi'(t'_\alpha,x'_\alpha)\right)\geq 0\label{e.xi'_sup}
\end{gather}
where
\begin{align*}
    \partial_t\phi(t_\alpha,x_\alpha) & = \delta +\delta  (T-t_\alpha)^{-2}  +\partial_t\zeta_2(t_\alpha,x_\alpha) +2\alpha(t_\alpha-t'_\alpha),
    \\
    \partial_t\phi'(t'_\alpha,x'_\alpha) & =-\delta(T-t'_\alpha)^{-2}  + 2\alpha(t_\alpha-t'_\alpha),
    \\
    \nabla \phi(t_\alpha,x_\alpha) & = \nabla \zeta_2(t_\alpha,x_\alpha)-2\alpha(x'_\alpha-x_\alpha-y_\alpha)+2\delta(x_\alpha-x'_\alpha)
    \\
    \nabla \phi'(t'_\alpha,x'_\alpha) & = -2\alpha(x'_\alpha-x_\alpha-y_\alpha)+2\delta(x_\alpha-x'_\alpha).
\end{align*}
Using \eqref{e.d_tzeta_2>}, we have
\begin{align}\label{e.rel_der_xi}
    \partial_t\phi(t_\alpha,x_\alpha) - \partial_t\phi'(t'_\alpha,x'_\alpha) \geq \delta + V|\nabla \phi(t_\alpha,x_\alpha) - \nabla \phi'(t'_\alpha,x'_\alpha)|.
\end{align}
Recall from its definition in \eqref{e.V} that $V$ is the Lipschitz coefficient of $\sF$ over a ball of radius $L+K+LK$.
Assuming that
\begin{align}\label{e.bd_nabla_phi}
    |\nabla \phi(t_\alpha,x_\alpha)|,\, |\nabla \phi'(t'_\alpha,x'_\alpha)| \leq L+K+LK,
\end{align}
we can deduce from \eqref{e.xi_sub}, \eqref{e.xi'_sup} and \eqref{e.rel_der_xi} that $\delta\leq 0$, which is absurd.

\textit{Step~5.}
We complete the proof by verifying \eqref{e.bd_nabla_phi}. For brevity, we write $\nabla\phi=\nabla\phi(t_\alpha,x_\alpha)$ and $\nabla\phi'=\nabla\phi'(t'_\alpha,x'_\alpha)$.

We first treat $\nabla\phi$. Since $f-\phi$ achieves a local maximum at $(t_\alpha,x_\alpha)$, we have $f(t_\alpha,x) - f(t_\alpha,x_\alpha) \leq \phi(t_\alpha,x) - \phi(t_\alpha,x_\alpha)$ for $x\in\C$ sufficiently close to $x_\alpha$. Note that for every $\lambda>0$ and every $z\in\C$, we have $\eps\lambda(z-x_\alpha) +x_\alpha\in\C$ for all $\eps \in (0, \frac{1}{\lambda}]$. Replacing $x$ by $\eps\lambda(z-x_\alpha) +x_\alpha$ and sending $\eps\to0$, by \eqref{e.Llowerbd} we have
\begin{align}\label{e.>-L}
    \la\lambda (z-x_\alpha) ,\nabla\phi\ra \geq -L|\lambda (z-x_\alpha)|,\quad \forall \lambda>0, z\in\C.
\end{align}
Since
\begin{align*}
    \nabla \zeta_2(t_\alpha,x_\alpha) = \frac{\theta'\left(\left(1+|x_\alpha|^2\right)^\frac{1}{2}+Vt_\alpha-R\right)}{\left(1+|x_\alpha|^2\right)^\frac{1}{2}} x_\alpha,
\end{align*}
we can see that $-\nabla\phi - u_\alpha = \lambda_\alpha(z_\alpha-x_\alpha)$ for some $\lambda_\alpha>0$ and $z_\alpha\in\C$, where for brevity we have set
\begin{align*}
    u_\alpha = 2\alpha(x'_\alpha -x_\alpha-y_\alpha).
\end{align*}
Hence, inserting $\lambda_\alpha, z_\alpha$ into \eqref{e.>-L}, we get
\begin{align*}
    \la -\nabla\phi - u_\alpha, \nabla\phi\ra \geq -L |\nabla\phi + u_\alpha|.
\end{align*}
Using $|u_\alpha|\leq K$ due to \eqref{e.<K}, we have
\begin{align*}
    |\nabla\phi|^2\leq (L+K)|\nabla\phi|+ LK
\end{align*}
and thus
\begin{align*}
    |\nabla\phi|\leq 1\vee(L+K+LK)= L+K+LK,
\end{align*}
where in the last equality we used our choice that $L> 1$.
Now, we turn to $|\nabla\phi'|$. Since $f-\phi'$ achieves a local minimum at $(t'_\alpha,x'_\alpha)$, by a similar argument, we have
\begin{align*}
    \la \lambda(z-x'_\alpha) ,\nabla\phi' \ra \leq L |\lambda(z-x'_\alpha)|,\quad\forall \lambda>0,z\in\C.
\end{align*}
Since now we simply have $\nabla\phi'+u_\alpha = 2\alpha(x_\alpha-x'_\alpha)$, we deduce that
\begin{align*}
    \la \nabla\phi'+u_\alpha, \nabla\phi'\ra \leq L |\nabla\phi'+u_\alpha|
\end{align*}
and, due to $|u_\alpha|\leq K$ (see \eqref{e.<K}),
\begin{align*}
    |\nabla\phi'|\leq 1\vee (L+K+LK)=L+K+LK.
\end{align*}
Therefore, we have verified \eqref{e.bd_nabla_phi} and completed the proof.
\end{proof}

\bibliographystyle{abbrv}

\end{document}